\documentclass[a4paper, 12 pt, reqno]{amsart}
\usepackage{latexsym,amssymb,amscd,amsbsy,amsfonts,bbm}
\usepackage{enumerate,verbatim}
\usepackage{color}

\usepackage[dvips]{graphicx} 
\usepackage[applemac]{inputenc}
\usepackage[T1]{fontenc}
\usepackage{lmodern}
\usepackage{textcomp}
\usepackage[margin=1in]{geometry}
\newtheorem{theorem}{Theorem}[section]
\newtheorem{lemma}[theorem]{Lemma}

\newtheorem{corollary}[theorem]{Corollary}
\newtheorem{proposition}[theorem]{Proposition}
\theoremstyle{definition}
\newtheorem{problem}[theorem]{Problem}
\newtheorem{definition}[theorem]{Definition}
\newtheorem{example}[theorem]{Example}

\newtheorem{remark}[theorem]{Remark}
\newtheorem{remarks}[theorem]{Remarks}
\numberwithin{equation}{section}

\newcommand{\N}{\mathbf N}
\newcommand{\Z}{\mathbf Z}
\newcommand{\R}{\mathbf R}
\newcommand{\Q}{\mathbf Q}

\newcommand{\mk}{\mathfrak}

\newcommand{\ab}{\operatorname{ab}}

\newcommand{\can}{\operatorname{can}}

\newcommand{\CB}{\operatorname{CB}}
\newcommand{\Cond}{\operatorname{Cond}}

\newcommand{\GL}{\operatorname{GL}}
\newcommand{\gp}{\operatorname{gp}}
\newcommand{\Hom}{\operatorname{Hom}}

\newcommand{\id}{\mathbbm{1}}

\newcommand{\SL}{\operatorname{SL}}
\newcommand{\EL}{\operatorname{EL}}
\DeclareMathOperator{\res}{res}

\newcommand{\HNN}{\operatorname{HNN}}

\newcommand{\ZZ}{\mathcal{Z}}

\def\DD{\mathcal{D}}
\def\FF{\mathcal{F}}
\def\GG{\mathcal{G}}

\def\NN{\mathcal{N}}
\def\OO{\mathcal{O}}
\def\RR{\mathcal{R}}
\def\VV{\mathcal{V}}

\newcommand{\mono}{\rightarrowtail}
\newcommand{\incl}{\hookrightarrow}
\newcommand{\epi}{\twoheadrightarrow}
\newcommand{\iso}{\stackrel{\sim}{\longrightarrow}}

\title{Infinite presentability of groups and condensation}
\author[Bieri, Cornulier, Guyot, \and Strebel]
{Robert Bieri, Yves Cornulier, Luc Guyot, \and Ralph Strebel}

\address{Department of Mathematics\\ Johann Wolfgang Goethe-Universit{\"a}t Frankfurt\\
60054 Frankfurt am Main, Germany}
\email{bieri@math.uni-frankfurt.de}

\address{Laboratoire de Math\'ematiques\\
B\^atiment 425, Universit\'e Paris-Sud 11\\
91405 Orsay\\France}
\email{yves.cornulier@math.u-psud.fr}

\address{STI, EPFL, INN 238 Station 14\\
 Lausanne 1015\\ Switzerland}
\email{luc.guyot@cyberbotics.com}

\address{D\'epartement de Math\'ematiques\\
Chemin du Mus\'ee 23\\ Universit\'e de Fribourg\\ 1700 Fribourg\\ Switzerland}
\email{ralph.strebel@unifr.ch}

\subjclass[2010]{Primary 20F05, Secondary 20E05, 20E06, 20F16, 20J06}
\keywords{Cantor-Bendixson rank, condensation group, infinitely presented metabelian group,
invariant Sigma, Thompson's group $F$, space of marked groups}

\thanks{L.\ G.\ thanks the Max Planck Institute for Mathematics in Bonn for the excellent conditions provided for his stay at this institution, during which part of the paper was written.}
\date{September 9, 2013}
\begin{document}
%

\begin{abstract}
We describe various classes of infinitely presented groups
that are condensation points in the space of marked groups. 
A well-known class of such groups consists of finitely generated groups 
admitting an infinite minimal presentation.
We introduce here a larger class of condensation groups,
called infinitely independently presentable groups, 
and establish criteria 
which allow one to infer that a group is infinitely independently presentable.
In addition,
we construct examples of finitely generated groups with no minimal presentation, 
among them infinitely presented groups with Cantor-Bendixson rank 1,
and we prove that every infinitely presented metabelian group is a condensation group. 
\end{abstract}

\maketitle
%
%
\section{Introduction} \label{intro}
In \cite{Neu37}, B. H. Neumann raises the question
whether there are finitely generated groups which cannot be defined by a finite set of relations.
Neumann asserts that such groups exist and offers two justifications.
One is based in the fact that there are $2^{\aleph_0}$ finitely generated groups
no two of which are isomorphic (\cite[Thm.\,14]{Neu37},
but only countably many isomorphism types of groups with a finite presentation. 
The second justification consists in the construction of a group $G = F/R$, 
defined by 2 generators and an infinite sequence of relators,
and the verification that no relator is contained in the normal subgroup generated by the remaining relators
(\cite[Thm.\,13]{Neu37}).
In the sequel, a presentation will be said to be \emph{minimal} 
if none of its relators is a consequence of the remaining ones
and a group will be called \emph{infinitely presented} if it does not admit a finite presentation.
A finitely generated group $G$ which admits a minimal presentation with an infinite set of relators 
is an infinitely presented group but the converse need not hold (see, e.\.g., Proposition 4.10).
Neumann's second justification uses therefore a condition 
which is sufficient, but not necessary,  for $G$ to be infinitely presented.
There are other conditions of this kind, 
in particular the following two: 
the hypothesis that the multiplier $H_2(G, \Z)$ of the finitely generated group be infinitely generated
and the assumption that $G$ map onto $\Z$, say $\pi \colon G \epi \Z$,
have no non-abelian free subgroup and be not an ascending HNN-extension with finitely generated base group contained in $\ker \pi$ (cf.\,\cite[Thm.\,A]{BiSt80}).

In this paper we are concerned with the question
whether a finitely generated group satisfying one of the stated sufficient conditions 
enjoys other interesting properties besides being infinitely presented.
We shall see, in particular,  
that the group is then often a \emph{condensation group},
i.\,e., a condensation point in the space of marked groups.
%

%
%
\subsection{Infinitely independently presented groups}
\label{ssec:INIP}
%
Key notions of this paper are the concept of an infinite independent family of normal subgroups
and the related concept of an infinitely independently presented group.
\begin{definition}
\label{definition:independent-family}
Let $G$ be a group 
and $\{N_i  \mid i\in I\}$ a family of normal subgroups of $G$. 
We say this family  is \emph{independent} if none of the normal subgroups $N_i$ is redundant,
i.\,e., contained in the subgroup generated by the remaining normal subgroups.
\end{definition}

\begin{definition}
\label{definition:INIP}
Let $G$ be a finitely generated group.\
We say that $G$ is \emph{infinitely independently presentable} (INIP for short) if there exists an epimorphism
$\pi \colon F \epi G$ of a free group of finite rank onto $G$ 
such that the kernel $R$ of $\pi$ is generated by an infinite independent family $\{R_i \mid i\in I\}$ of normal subgroups $R_i \triangleleft F$.
\end{definition} 

An infinitely minimally presented group yields an infinite independent family of relators in the free group 
and hence an infinite independent family of normal subgroups. 
Finitely generated groups admitting an infinite minimal presentation are therefore INIP. 
Another class of infinitely independently presentable groups is provided by   
\begin{proposition}[Prop.\ \ref{p_schur}]
\label{prp:Infinitely-independently-presented-groups-and-Schur-multiplier}
Let $G$ be a finitely generated group 
and assume that the Schur multiplier $H_2(G,\Z)$ is generated by an infinite independent family of subgroups. 
Then $G$ is infinitely independently presentable.
\end{proposition}
%

%
\subsection{The space $\GG_m$ of marked groups}
\label{ssec:Space-marked-groups}
%
Let $m$ be a positive integer.
A \emph{group marked by the index set }$I_m =\{1, 2, \ldots, m\}$ is a group $G$,
together with a map $\iota_{G} \colon I_m \to G$ whose image generates $G$.
Two marked groups  $(G, \iota_G)$ and $(H, \iota_H)$ are said to be isomorphic 
if there is an isomorphism $\varphi \colon G \iso H$ satisfying the relation $\iota_H = \varphi \circ \iota_G$. 

We denote by $\GG_m$ the set of groups marked by $I_m$,  
up to isomorphism of marked groups.

It is useful to have an explicit model of the set $\GG_m$.
To that end, we introduce the free group $F_m$ with basis $I_m$ 
and the marked group $(F_m, \iota_{F_m})$ 
where $\iota_{F_m} \colon I_m \incl F_m$ denotes the obvious embedding.

Given a group $(G, \iota_G)$ marked by $I_m$
there exists then a unique  epimorphism $\rho$ of the free group $F_m$ onto $G$ 
satisfying the relation $\rho \circ \iota_{F_m} = \iota_G$.
The kernel $R$ of $\rho$ does not change if the marked group $(G,\iota_G)$
is replaced by an isomorphic copy. 
The set $\GG_m$ is thus parameterized by the normal subgroups of $F_m$.
\smallskip

The space $\GG_m$ carries a natural topology 
which goes back to Chabauty \cite{Cha50} (see \cite{Gri05}, \cite[Sec.~2.2]{ChGu05} or \cite[p.~258]{CGP07} for more details). 
It has a convenient description where $\GG_m$ is identified with the set of normal subgroups of $F_m$.
Then the collection of all sets of the form
\begin{equation}
\label{eq:Basis-topology}
\OO_{\FF, \FF'} = \{ S \triangleleft F_m \mid \FF \subseteq S \text{ and } S \cap \FF' = \emptyset \},
\end{equation}
where $\FF$ and $\FF'$ range over {\em finite} subsets of $F_m$, is a basis of the topology on $\GG_m$. 

\subsection{Cantor-Bendixson rank of a finitely generated group}
\label{ssec:Cantor-Bendixson-rank}
%
If $T$ is a topological space,  
its subset of accumulation points gives rise to a subspace $T'$.
We now apply this construction to the topological space  $X= \GG_m$ and iterate it.
The outcome is a family of subspaces;
it starts with
\[
X^{(0)} = X, \quad X^{(1)} = X', \ldots, X^{(n+1)} = {X^{(n)}}',  \ldots, X^{(\omega)}  =
\bigcap\nolimits_n X^{(n)}, 
\]
and is continued by transfinite induction.
This family induces a function on $\GG_m$ with ordinal values,
called \emph{Cantor-Bendixson rank} or CB-rank for short 
(see \S\ref{sec:Elements-CB-theory} for a more detailed description).

The points with CB-rank 0 are nothing but the isolated points of $X$;
the points with CB-rank 1 are the points which are not isolated, but are isolated among non-isolated points.
At the other extreme, are the points 
which lie in the intersection of all $X^{(\alpha)}$; they are called \emph{condensation points}.

A finitely generated group $G$ is said to have CB-rank $\alpha$ 
if it occurs as a point of CB-rank $\alpha$ in a space $\GG_m$ for some integer $m$.
It is a remarkable fact that this rank does not depend on the generating system $\iota \colon I_m \to G$
and, in particular, not on $m$ (see part (3) of \cite[Lem.~1.3]{CGP07}).
The CB-rank is thus a group theoretic property;
it allows one to talk about ``isolated groups'' and ``condensation groups''. A sufficient condition for a finitely generated group $Q$ to be both condensation and infinitely presented is to be of {\em extrinsic condensation}, in the sense that for every finitely presentable group $G$ and epimorphism $G\epi Q$, the kernel $K$ has uncountably many subgroups that are normal in $G$. 

\subsection{Condensation groups}
\label{ssec:Condensation-groups}
One goal of this paper is to exhibit various classes of condensation groups.
To put these classes into perspective,
we begin with a few words on groups of Cantor-Bendixson rank 0 or 1.

\subsubsection{Isolated groups and groups of CB-rank 1}
\label{sssec:Isolated-groups-CB-rank-1}
%
Isolated groups can be characterized in a simple manner. 
The following result is a restatement of \cite[Prop.~2(a)]{Man82}. 
Its topological interpretation was obtained independently in \cite[Thm.~2.1]{Gri05} and \cite[Prop.\ 2.2]{CGP07}.

\begin{proposition}
\label{prp:Characterization-isolated-group}
A finitely generated group is isolated if, and only if, it admits a finite presentation 
and has a finite collection $(M_j)$ of nontrivial normal subgroups 
such that every nontrivial normal subgroup contains one of them.
\end{proposition}
Obvious examples of isolated groups are therefore finite groups
and finitely presentable simple groups.
But there exists many other isolated groups (see \cite[Sec.~5]{CGP07}). 
\smallskip

Proposition \ref{prp:Characterization-isolated-group} contains the observation that every isolated group admits a finite presentation. 
Actually, it is easier to detect finitely presentable groups that are non-condensation 
than infinitely presented groups with this property.
Indeed,  if a group $G$ comes with an epimorphism $\rho \colon F_m \epi G$, and if 
the kernel $R$ of $\rho$ is the normal closure of a finite set, $\RR$ say, then
$$\OO_{\RR,\emptyset}=\{ S\triangleleft F_m \mid  R \subseteq S\}$$
is a neighbourhood of $G$. Thus, if $G$ is finitely presentable, then a basis of neighbourhoods of $G$ can be described as
\begin{equation}
\label{eq:Neighbourhoods-fr-group}
\OO^G_{\FF'} = \{ S\triangleleft G \mid  S \cap \FF' =  \emptyset \},
\end{equation}
where $\FF'$ ranges over finite subsets of $G \smallsetminus\{1\}$. 
The latter description does not refer to any free group mapping onto $G$.

The previous remark implies that an infinite cyclic group and, more generally,  
a finitely presentable, infinite residually finite, just-infinite group has CB-rank 1. A simple instance \cite{McC68} of such a group is consider, for any $n\ge 2$, the semidirect product $\Z^n_0\rtimes\textnormal{Sym}(n)$, where $\Z^n_0$ is the subgroup of $\Z^n$ of $n$-tuples with zero sum, and $\textnormal{Sym}(n)$ is the symmetric group, acting on $\Z^n$ by permutation of coordinates. A more elaborate example, due to Mennicke \cite{Men65}, is $\textnormal{PSL}_n(\Z)$ if $n\ge 3$.
By contrast, most infinitely presented groups are known or expected to be condensation and in particular no well-known infinitely presented group seems to have CB-rank 1;
but such groups exist as can be seen from
\begin{theorem}[Thm.\ \ref{ThCBone}] 
\label{thm:Infinitely-related-CB-rank-1}
There exist infinitely presented groups with Cantor-Bendixson rank 1. Moreover, they can be chosen to be nilpotent-by-abelian.
\end{theorem}

\subsubsection{Criteria}
\label{sssec:Criteria-condensation-groups}
%
It is an elementary fact 
that INIP groups are infinitely presented extrinsic condensation groups 
(Remark \ref{rem:INIP-implies-extrinsic-condensation}); 
so the INIP criteria in Section \ref{secip}, in particular  \ref{p_schur}, are also condensation criteria.
Four other criteria for condensation groups will be established.
The first of them covers both finitely presentable and infinitely presented groups: 
\begin{proposition}[Cor.\ \ref{corols}]
\label{prp:Main-criterion-intrinsic-condensation-groups}
Every finitely generated group with a normal non-abelian free subgroup is a condensation group.
\end{proposition}
The proposition is actually an elementary consequence of the following difficult result independently due to Adian, Olshanskii, and Vaughan-Lee (see \S\ref{aovl}): any non-abelian free group has uncountably many characteristic subgroups. As a simple application, the non-solvable Baumslag-Solitar groups
\[
\textnormal{BS}(m,n)=\langle t,x \mid tx^mt^{-1}=x^n\rangle\qquad (|m|,|n|\ge 2)
\]
are condensation groups. 

Proposition \ref{prp:Main-criterion-intrinsic-condensation-groups} 
will be a a stepping stone in the proof of
\begin{theorem}[Cor.\ \ref{crl:Corollary-for-ii}]%
\label{thm:Non-contraction}
Let $G$ be a finitely generated group with an epimorphism $\pi:G\epi\Z$. 
Suppose that $\pi$ does not split over a finitely generated subgroup of $G$, 
i.e.\, there is no decomposition of $G$ as an HNN-extension over a finitely generated subgroup 
for which the associated epimorphism to $\Z$ is equal to $\pi$.
Then $G$ is an extrinsic condensation group (and is thus infinitely presented).
\end{theorem}

Theorem \ref{thm:Non-contraction} is based on Theorem \ref{thm:Structure-theorem} which is an extension of \cite[Thm.~A]{BiSt78}; the latter states that every homomorphism from a finitely presentable group onto the group $\Z$ splits over a finitely generated subgroup.

\subsubsection{Link with the geometric invariant}%
\label{sssec:Link-with-Sigma1}
If $G$ is not isomorphic to a non-ascending HNN-extension over any finitely generated subgroup, Theorem \ref{thm:Non-contraction} can be conveniently restated in terms of the {\em geometric invariant} $\Sigma(G)=\Sigma_{G'}(G)$ introduced in \cite{BNS87}. We shall view this invariant as an open subset of the real vector space $\Hom(G,\R)$ (see \S\ref{subns} for more details). Writing $\Sigma^c(G)$ for the complement of $\Sigma(G)$ in $\Hom(G,\R)$, we have

\begin{corollary}[Cor.\ \ref{crl:Criterion-for-extrinsic-condensation}]
\label{crl:Non-ascending-HNN-extension-and-condensation-groups}
 Let $G$ be a finitely generated group and assume that $G$ is not isomorphic to a non-ascending HNN-extension over any finitely generated subgroup (e.g., $G$ has no non-abelian free subgroup).
 If $\Sigma^c(G)$ contains a rational line 
 then $G$ is an extrinsic condensation group (and is thus infinitely presented). 
 \end{corollary}
 This corollary applies to several classes of groups $G$ where  $\Sigma^c(G)$ is known,
 such as finitely generated metabelian groups or groups of piecewise linear homeomorphisms.

Corollary \ref{crl:Non-ascending-HNN-extension-and-condensation-groups} and results about the geometric invariant of metabelian groups proved in \cite{BiSt80} and \cite{BiGr84}, allow one to give a description of the finitely generated 
centre-by-metabelian groups that are condensation groups:
\begin{corollary}[Cor.\ \ref{crl:Fp-centre-by-metabelian-groups}]
\label{crl:Center-by-metabelian-condensation-groups}
A finitely generated centre-by-metabelian group is a condensation group if, and only if, it is infinitely presented.
\end{corollary}

The Cantor-Bendixson rank of finitely presentable metabelian groups was computed in \cite{Cor11b}. 
Corollary \ref{crl:Center-by-metabelian-condensation-groups} extends this computation to the infinitely presented ones while
answering \cite[Qst.~1.7]{Cor11b} .

The second principal application of Corollary \ref{crl:Non-ascending-HNN-extension-and-condensation-groups} 
deals with groups of piecewise linear homeomorphisms of the real line.
For some of these groups, the $\Sigma$-invariant has been computed in \cite[Thm.\;8.1]{BNS87}. 
These computations apply, in particular, to Thompson's group $F$, 
the group consisting of all increasing piecewise linear homeomorphisms of the unit interval 
whose points of non-differentiability are dyadic rational numbers and whose slopes are integer powers of 2. 
We denote by  $\chi_0$ (resp. by $\chi_1$) the homomorphism from F to $\R$
which maps $f$ to the binary logarithm of its slope at 0 (resp. at 1). 
The homomorphisms $\chi_0$ and $\chi_1$ form a basis of $\Hom(F,\Z)$; 
so every normal subgroup of $F$ with infinite cyclic quotient is the kernel 
$N_{p,q} = \ker(p\chi_0 - q \chi_1)$, with ($(p,q) \in \Z^2 \smallsetminus \{(0,0)\}$. 
The discussion in \cite{BNS87} describes  
which of the groups $N_{p,q}$ are finitely generated, or finitely presentable, 
and how their invariants $\Sigma(N_{p,q})$ look like. 
Using Corollary \ref{crl:Non-ascending-HNN-extension-and-condensation-groups}  and methods from \cite{CFP96}, 
one can describe how the groups $N_{p,q}$ sit inside the space of marked groups.

\begin{corollary} 
\label{crl:Normal-subgroups-in-F}
\begin{enumerate}[(1)]
\item If $p\cdot q = 0$, then $N_{p,q}$ is infinitely generated,
\item  if $p\cdot q > 0$, then $N_{p,q}$  is finitely presentable and isolated,
\item  if $p\cdot q < 0$, then $N_{p,q}$ is a finitely generated extrinsic condensation group (and thus infinitely presented).  
\end{enumerate}
\end{corollary}
Here is a third application, where the group contains non-abelian free subgroups. 
\begin{corollary}[Ex.\ \ref{sla}]\label{isla}
Let $A=\Z[x_n:n\in\Z]$ be the polynomial ring in infinitely many variables. Consider the ring automorphism $\phi$ defined by $\phi(x_n)=x_{n+1}$, which induces a group automorphism of the matrix group $\SL_d(A)$. Then for all $d\ge 3$, the semidirect product $\SL_d(A)\rtimes_\phi\Z$ is finitely generated and of extrinsic condensation (and thus infinitely presented).
\end{corollary}

This leaves some open questions. Notably
\begin{problem}
Is any infinitely presented, finitely generated metabelian group infinitely independently presentable 
(see Definition \ref{definition:INIP})? 
Does it admit a minimal presentation? 
Here are three test-cases
\begin{itemize}
\item The group $\Z[1/p]^2\rtimes\Z$, with action by the diagonal matrix $(p,p^{-1})$. Here $p$ is a fixed prime number.
\item The group $\Z[1/pq]\rtimes\Z$, with action by multiplication by $p/q$. Here $p,q$ are distinct prime numbers.
\item The free metabelian group on $n\ge 2$ generators.
\end{itemize}
\end{problem}

\subsection*{Acknowledgments.}
We are grateful to Pierre de la Harpe for useful comments on several versions of this paper 
and for helpful references.
We are also grateful to Alexander Olshanskii 
for suggesting Proposition \ref{prp:Main-criterion-intrinsic-condensation-groups}, 
to Mark Sapir for Remark \ref{sapirem} 
and to Slava Grigorchuk for encouragement and suggestions.

\setcounter{tocdepth}{1}    
\tableofcontents
\newpage

%
\section{Preliminaries}
\subsection{Elements of the Cantor-Bendixson theory} 
\label{sec:Elements-CB-theory}
%
%
Let $X$ be a topological space. Its derived subspace $X^{(1)}$ is, by definition,  
the set of its accumulation points. Iterating over ordinals 
\[
X^{(0)}=X,\, X^{(\alpha+1)}=(X^{(\alpha)})^{(1)},\, X^{(\lambda)}=\bigcap_{\beta <\lambda} X^{(\beta)} 
\text{ for $\lambda$ a limit ordinal, }
\]
one constructs a non-increasing family $X^{(\alpha)}$ of closed subsets. 
If $x\in X$, we write 
\[
\CB_X(x) = \sup\{\alpha \mid x \in X^{(\alpha)}\}
\]
if this supremum exists, in which case it is a maximum. 
Otherwise we write $\CB_X(x) = \mk{C}$, where the symbol $\mk{C}$ is not an ordinal. 

We call $\CB_X(x)$ the \emph{Cantor-Bendixson rank of} $x$.
If $\CB_X(x) \neq \mk{C}$ for all $x \in X$, i.e.\ if $X^{(\alpha)}$ is empty for some ordinal, 
we say that $X$ is \emph{scattered}.
A topological space $X$ is called \emph{perfect} if it has no isolated point, i.e., $X^{(1)}=X$. 	
As a union of perfect subsets is perfect, 
every topological space has a unique largest perfect subset, 
called its \emph{condensation part} and denoted $\Cond(X)$. 
Clearly, $\Cond(X)$ is empty if and only if $X$ is scattered and we have 
\[
\Cond(X)=\{x \in X \mid \CB_X(x)=\mk{C}\}.
\]
The subset $X \setminus \Cond(X)$ is the largest scattered open subset, and is called the {\it scattered part} of $X$. 

In the sequel,
we apply the concepts recalled in the above, 
either to the space $\NN(G)$ of all normal subgroups of a finitely generated group $G$,
or to the space of marked groups $\GG_m$   (defined in  Section \ref{ssec:Space-marked-groups}).
Both spaces have three crucial properties:
they are \emph{compact} (by Tychonoff's theorem);
they are \emph{metrizable},
for $G$ and $F_m$ are finitely generated, and hence countable, groups;
and they are totally disconnected.
It follows from the first two properties that $\NN(G)$ and $\GG_m$ have countable bases, and hence countable scattered parts. The three properties, when combined together with perfectness, characterize Cantor spaces. 
As the cardinality of $\GG_m$ is $2^{\aleph_0}$ for $m > 1$ 
(\cite[Thm.~14]{Neu37}, cf.\  \cite[Thm.~7]{Hal54}), 
the condensation part of of $\GG_m$ is homeomorphic to the Cantor set 
(see \cite{Gri84} for an explicit embedding of the Cantor space into $\GG_m$).
This last fact entails, in particular, that every neighbourhood of a condensation point in $\GG_m$ contains $2^{\aleph_0}$ points.

\subsection{The Chabauty topology}

Let $X$ be a topological space, and $F(X)$ the set of closed subsets in $X$. If $(U_j)_{j\in J}$ is a finite family of open subsets of $X$, and $K$ is a compact subset in $X$, define
$$\mathcal{U}((U_j)_{j\in J},K)=\{M\in F(X)\mid M\cap K=\emptyset;\;\;M\cap U_j\neq\emptyset, \forall j\}.$$
Then the $\mathcal{U}((U_j)_{j\in J},K)$, when $(U_j)_{j\in J}$ ranges over all finite families of open subsets of $X$, and $K$ ranges over compact subsets of $X$, form a basis of a topology, called the Chabauty topology on $X$. (It is enough to consider the $U_j$ ranging over a given basis of open sets of $X$.) 

When $X$ is discrete, a basis of open sets is given by singletons. If $(U_j)$ is a finite family of open sets, its union $\mathcal{F}$ is finite and the condition $\forall j,U_j\cap M\neq\emptyset$ simply means $\mathcal{F}\subset M$. Thus a basis of the Chabauty topology of $F(X)=2^X$ is given by the
$$\mathcal{U}(\mathcal{F},\mathcal{F}')=\{M\subset X\mid \mathcal{F}\subset M,\;M\cap\mathcal{F}'=\emptyset\},$$
which looks like the way we introduce it in \S\ref{ssec:Space-marked-groups}.

If $X$ is a locally compact (Hausdorff) space, the space $F(X)$ is Hausdorff and compact.

When $X=G$ is a locally compact group, an elementary verification shows that the set $\mathcal{N}(G)$ of closed normal subgroups is closed in $F(X)$.

Here we only address the case of discrete groups, where the above verifications are even much easier. In this context, it was observed by Champetier-Guirardel \cite[\S 2.2]{ChGu05} that the topology on the set of marked groups on $d$ generators, as introduced by Grigorchuk \cite{Gri84}, 
coincides with the Chabauty topology on the set of normal subgroups of the free group on $d$ generators.

%
%
\section{Abels' groups and a group with Cantor-Bendixson rank 1} 
\label{sec:Group-with-CB-rank-1}
%
In this section, 
we prove Theorem \ref{thm:Infinitely-related-CB-rank-1}. 
%
\subsection{Abels' finitely presentable matrix groups} 
\label{ssec:Abels-groups}
%
Our construction of an infinitely presented group with Cantor-Bendixson rank 1 
is based on a sequence of groups studied by H. Abels and K. S. Brown (\cite{Abe79} and \cite{AbBr87}).
Let $n \ge 3$ be a natural number,
$p$ be a prime number and $\Z[1/p]$ the ring of rationals with denominator a power of $p$. 
Define $A_n \le \GL_n(\Z [1/p])$  
to be the group of matrices of the form 
\[
\begin{pmatrix} 1 & \star & \cdots &\star & \star \\
0 & \star & \cdots & \star&\star\\
\vdots & \vdots & \ddots & \vdots&\vdots\\
0 & 0 & \cdots & \star&\star\\
0 & 0 & \cdots & 0&1\\
\end{pmatrix}
\]
with integral powers of $p$ in the diagonal. 
The group $A_n$ is finitely generated for $n\ge 3$.
The group $A_3$, which was introduced by P.~Hall in \cite{Hal61}, is infinitely presented; for $n \geq 4$, however, the group $A_n$ has a finite presentation. 
Its centre $\mathcal{Z}(A_n)$ is isomorphic to the additive group of $\Z[1/p]$ and thus infinitely generated,
whence the quotient group $A_n/\mathcal{Z}(A_n)$ is  
infinitely presented for all $n\ge 3$.

\subsubsection{Normal subgroups in the groups $A_n$}
\label{sssec:Normal-subgroups-Abels-groups}
We want to show next that the Cantor-Bendixson rank of the quotient $A_n/\mathcal{Z}(A_n)$ is countable 
for $n \geq 4$.  
As a preliminary step, we investigate the normal subgroups of $A_n/\mathcal{Z}(A_n)$. 
For this we need further definitions. 
\begin{definition}
\label{definition:Omega-groups}
Let $\Omega$ be a group. 
A group $G$ endowed with a left action of $\Omega$ by group automorphisms 
is called an $\Omega$-\emph{group}; 
an $\Omega$-invariant subgroup $H$ of $G$ is called an $\Omega$-\emph{subgroup} of $G$.
 If $H$ is a normal subgroup of $G$ 
then $G/H$ is an $\Omega$-group for the induced $\Omega$-action.

An $\Omega$-group $G$ satisfies \emph{max-$\Omega$} if 
it admits no increasing sequence of $\Omega$-subgroups, or equivalently, 
if every $\Omega$-subgroup of $G$ is finitely generated as an $\Omega$-group. 
If $\Omega$ is the group of inner automorphisms of $G$, 
the $\Omega$-subgroups of $G$ are the normal subgroups of $G$. 
If $G$ satisfies max-$\Omega$ in this case, we say that $G$ satisfies \emph{max-n}. 
\end{definition}

We are now ready for the analysis of the normal subgroups of $A_n$ with $n \geq 3$. 
Since its centre $\mathcal{Z}(A_n) \cong \Z[1/p]$ is infinitely generated,
the group $A_n$ does not satisfy max-n. 
Still, we have
\begin{lemma} \label{LemMaxN}
The group $A_n/\mathcal{Z}(A_n)$ satisfies max-n.
\end{lemma}

\begin{proof}
We denote by $\id_n$ the $n$-by-$n$ identity matrix and by $E_{ij}$ the $n$-by-$n$ elementary matrix 
whose only non-zero entry is at position $(i,j)$ with value 1; here  $1 \le i<j \le n$.  
Set $U_{ij}=\id_n+\Z[1/p]E_{ij}$.

Let $D \simeq  \Z^{n-2}$ be the subgroup of diagonal matrices in $A_n$. 
It is enough to check that $A_n/\mathcal{Z}(A_n)$ satisfies max-$D$. 
Since the max-$D$ property is stable under extensions of $D$-groups (see the simple argument in \cite[3.7]{Rob96}), 
it suffices to find a subnormal series of $A_n$ whose successive subfactors satisfy max-$D$. 
This is easy: the group $D$  and the  unipotent one-parameter subgroups $U_{ij}$ for $i<j$ 
are isomorphic to factors  of a subnormal series of $A_n$ and these factors satisfy max-$D$,
except for the central subgroup $U_{1n}$.
(Note that the additive group of $U_{ij}$ is isomorphic to $\Z[1/p]$, 
and that some element of $D$ acts on it by multiplication by $p$ if $(i,j)\neq (1,n)$.)
\end{proof}

We now proceed to count subgroups of Abels' groups. We begin by a general lemma, which will also be used in Example \ref{cnine}.

\begin{lemma}\label{lemcos}
Let $G$ be a countable group, in an extension $N\mono G\epi Q$. Suppose that $N$ has at most countably many $G$-subgroups (that is, subgroups that are normal in $G$), and that $Q$ satisfies max-n. Then $G$ has at most countably many normal subgroups.
\end{lemma}
\begin{proof}
Because of the assumption on $N$, it is enough to show that for every normal subgroup $M_0$ of $G$ contained in $N$, there are at most countably many normal subgroups $M$ of $G$ such that $M\cap N=M_0$. Indeed, since $M/M_0$ is isomorphic, as a $G$-group, to a subgroup of $Q$, we see that it is finitely generated qua normal subgroup of $M/M_0$; in other words, $M$ is generated, as a normal subgroup, by the union of $M_0$ with a finite set. Thus there are at most countably many possibilities for $M$.
\end{proof}

\begin{proposition}\cite[Thm.~1]{Lyu84} \label{CorCountQuot}
The group $A_n$ has exactly countably many normal subgroups, although it does not satisfy max-n.
\end{proposition}
\begin{proof}
The centre $\mathcal{Z}(A_n)$ is isomorphic to $\Z[1/p]$ which is not finitely generated; therefore $A_n$ cannot satisfy max-n. As $\mathcal{Z}(A_n)\simeq\Z[1/p]$ has only countably many subgroups, and $A_n/\mathcal{Z}(A_n)$ satisfies max-n, Lemma \ref{lemcos} shows that $A_n$ has at most countably many normal subgroups.
\end{proof}

Proposition \ref{CorCountQuot} was only stated for $n=4$ in \cite{Lyu84} but the proof carries over to the general case with no significant modification. However, we need a more precise description of the set of normal subgroups of $A_n$, which readily implies Proposition \ref{CorCountQuot}.

\begin{lemma} \label{LemNormAbels}
Let $N$ be a normal subgroup of $A_n$. Then either
\begin{itemize}
\item $N\subset \mathcal{Z}(A_n)$, or \item $N$ is finitely generated as a normal
subgroup, and contains a finite index subgroup of $\mathcal{Z}(A_n)$.\end{itemize}
\end{lemma}

We denote by $U(A_n)$ the subgroup of unipotent matrices of $A_n$.

\begin{proof}
Suppose that $N$ is not contained in $\mathcal{Z}(A_n)$. Set $M=N\cap U(A_n)$.
We first prove that $Z'=M \cap \mathcal{Z}(A_n)$ has finite index in $\mathcal{Z}(A_n)$. 

 The image of $N$ inside $A_n/\mathcal{Z}(A_n)$ cannot 
intersect trivially $K=U(A_n)/\mathcal{Z}(A_n)$ as it is a non-trivial normal subgroup and since $K$ contains its own centralizer. Consequently $M$ is not contained in $\mathcal{Z}(A_n)$. Since $K$ is a non-trivial nilpotent group, the image $\bar{M}$ of $M$ in $K$ intersects $\mathcal{Z}(K)$ non-trivially (if an element of $\bar{M}$ is not central, perform the commutator with an element of $K$ to get another non-trivial element of $\bar{M}$, and reiterate until we obtain a central element: the process stops by nilpotency).
 Thus $M$ contains a matrix $m$ of the form $\mathbf{1}_n+r_1E_{1,n-1}+ r_2E_{2,n}+cE_{1,n}$ 
where one of the $r_i$ is not zero. Taking the commutators of $m$ with $U_{n-1,n}$ 
(resp. $U_{1,2}$) if $r_1 \neq 0$ (resp. $r_2 \neq 0$), we obtain a finite index subgroup $Z'$ of $\mathcal{Z}(A_n)=U_{1,n}$ which lies in $M$. The proof of the claim is then complete.

Now $A_n/Z'$ satisfies max-n by Lemma \ref{LemMaxN}, and therefore the image of $N$ in $A_n/Z'$ is finitely generated as a normal subgroup. Lift finitely many generators to elements generating a finitely generated normal subgroup $N'$ of $A_n$ contained in $N$. As $N$ is not contained in $\mathcal{Z}(A_n)$, $N'$ cannot be contained in $\mathcal{Z}(A_n)$. The claim above, applied to $N'$, shows that $N'$ contains a finite index subgroup of $\mathcal{Z}(A_n)$. 
As the index of $N'$ in $N$ coincides with the index of $N' \cap Z'$ in $Z'$, the former is finite. Therefore $N$ is finitely generated as a normal subgroup.
\end{proof}

\subsubsection{Cantor-Bendixson rank of  the groups $A_n/\mathcal{Z}(A_n)$}
\label{sssec:CB-rank-Abels-quotient-groups}
We get in turn

\begin{corollary} \label{crl:Countable-CB-rank}
For every $n\ge 4$ the quotient group of $A_n$
by its centre is an infinitely presented group with at most countable Cantor-Bendixson rank. In particular (anticipating on Lemma \ref{indcon}(\ref{ind1})), it is not INIP and admits no minimal presentation over any finite generating set.
\end{corollary}

\begin{proof}
Since $A_n$ is finitely presentable 
\cite{AbBr87} (\cite{Abe79} for $n=4$), 
there is an open neighborhood of $A_n/\mathcal{Z}(A_n)$ in the space of finitely generated groups 
which consists of marked quotients of $A_n$. 
This neighborhood is countable by Proposition \ref{CorCountQuot}
and so $A_n/\mathcal{Z}(A_n)$ is not a condensation group. But if so, its  Cantor-Bendixson rank must be a countable ordinal
(cf. the last part of \S\ref{sec:Elements-CB-theory}).
 \end{proof}
\begin{remark}
\label{remark:CB-rank-Abels-groups}
The Cantor-Bendixson rank of $A_n/\mathcal{Z}(A_n)$ for $n\ge 4$ can be computed explicitly:
it is $n(n+1)/2-3$ and so coincides with the number of relevant coefficients in a ``matrix'' in $A_n$, 
viewed modulo $\mathcal{Z}(A_n)$.  For $n=4$, this number  is 7. 
Similarly, the Cantor-Bendixson rank of $A_n$ is $n(n+1)/2-2$ for $n\ge 4$.
These claims can be justified as in \cite[Lem.~3.19]{Cor11b}.
We shall not give any details of the verification; 
instead we refine the construction in the next paragraph to get an example of an infinitely presented solvable group with Cantor-Bendixson rank one. 
\end{remark}

\subsection{Construction of a group with Cantor-Bendixson rank 1} 
\label{ssec:Construction-group-CB-rank-1}

We begin by explaining informally the strategy of our construction. As we saw, if $Z\simeq\Z[1/p]$ is the centre of $A_n$ and $n\ge 4$, then $A_n/Z$ is an infinitely presented group and is not condensation; this is based on the property that $Z$, albeit infinitely generated, has, in a certain sense few subgroups. However, $A_n/Z$ is not of Cantor-Bendixson rank one, notably because it has too many quotients. 

We are going to construct a similar example, but using an artifact to obtain a group that cannot be approximated by its own proper quotients. We will start with a group $A$, very similar to $A_5$ but whose centre $Z$ is isomorphic to $\Z[1/p]^2$, and $Z_0\simeq\Z^2$ a free abelian subgroup of rank two. 

We will consider a certain, carefully chosen subgroup $V_1$ of $Z$, containing $Z_0$, such that $Z/V_1$ and $V_1/Z_0$ are both isomorphic to $\Z(p^\infty)$. 
The factor group $A/V_1$ is infinitely presented 
and the quotient groups of the finitely presentable group $A/Z_0$ 
make up an open neighbourhood $\VV_0$ of $A/V_1$  in the space of marked groups.
If there existed a smaller neighbourhood $\VV$ of $V_1$ 
containing only subgroups of $Z/Z_0$ and in which $V_1/Z_0$ were the only infinite subgroup 
then $A/V_1$ would have have Cantor-Bendixson rank 1. At this point this seems illusory, because it can be checked elementarily that $V_1$ is a condensation point in the space of subgroups of $Z$. To remedy these shortcomings, we work in a semidirect product $A\rtimes C$, where $C$ is the cyclic subgroup generated by some automorphism normalizing $V_1$. The point is that the construction will ensure that only few of the subgroups of $Z$ will be normalized by $C$.

Both $C$ and $V_1$ have to be chosen with care. We will prescribe the automorphism generating $C$ to act on $Z\simeq\Z[1/p]^2$ in such a way that $\Z^2$ is invariant and that the action is diagonalizable over $\Q_p$ but not over $\Q$. If $\mathcal{D}_1$ is an eigenline in $\Q_p^2$, the subgroup $V_1$ will be defined as $(\mathcal{D}_1+\Z_p^2)\cap\Z[1/p]^2$. The assumption of non-diagonal\-izability over $\Q$ will intervene when we need to ensure that $C$ normalizes few enough subgroups in $A$.
%
\subsubsection{Choice of the group $A$}
\label{sssec:Choice-group-tilde-A}
Let $p$ be a fixed prime number 
and let $A$  be the subgroup of $A_5 \leq \GL_5(\Z[1/p])$ consisting of all matrices of the form
\[
\begin{pmatrix} 1 & \star & \star &\star & \star \\
0 & \star & \star & \star&\star\\
0 & 0& \star & \star&\star\\
0 & 0 & 0 & 1&0\\
0 & 0 & 0 & 0&1
\end{pmatrix}.
\]
The group $A$ is an $S$-arithmetic subgroup of $\GL_5(\Q)$ 
and so the question whether it admits a finite presentation can,
by the theory developed by H. Abels in \cite{Abe87}, be reduced to an explicit computation of homology, which in the present case is performed in \cite[\S 2.2]{dCor09}:
 \begin{lemma}
\label{Afp}
The group $A$ is finitely presentable.
\end{lemma}

\subsubsection{Choice of the automorphism and the subgroup $E$.}
\label{sssec:Choice-E}
%
The automorphism will be induced by conjugation by a matrix of the form
\[
M_1=\begin{pmatrix}\id_3 & 0 \\ 0 & M_0\end{pmatrix}\in\GL_5(\Z),
\]
where $M_0$ is any matrix in $\GL_2(\Z)$ satisfying the following requirements:
\begin{enumerate}[(a)]
\item\label{iti} $M_0$ is not diagonalizable over $\Q$;
\item\label{itii} $M_0$ is diagonalizable over the $p$-adic numbers $\Q_p$.
\end{enumerate} 

An example of a matrix $M_0$ with properties (\ref{iti}) and  (\ref{itii}) 
is the companion matrix of the polynomial $X^2+p^3X-1$:
one verifies easily that (\ref{iti}) holds for every choice of $p$;
to check (\ref{itii}), one can use \cite[Sec.~II.2]{Ser77b} to prove 
that the root 1 of this polynomial in $\Z/p^3 \Z$ lifts to a root in the ring of $p$-adic integers $\Z_p$. Note that $M$ is allowed to be a matrix of finite order (thus of order 3, 4, or 6 by (\ref{iti}), the validity of (\ref{itii}) depending on $p$).
\smallskip

Let $\lambda_1$ and $\lambda_2$ be the eigenvalues of the matrix $M_0$ in $\Q_p$.
Since multiplication by $M_0$ induces an automorphism of $\Z_p^2$, 
these eigenvalues are units of the ring $\Z_p$;
as $M_0$ cannot be a multiple of the identity matrix in view of requirements (\ref{iti}) and (\ref{itii}), they are distinct.
Let $\DD_1$ and $\DD_2$ denote the corresponding eigenlines in the vector space $\Q_p^2$.
Then $\Q_p^2=\DD_1 \oplus \DD_2$. 
By assumption, the intersection of $\DD_i$ ($i=1,2$) with the dense subgroup $\Z[1/p]^2$ are trivial so we rather consider the intersections
$$E_i=\Z[1/p]^2\cap (\DD_i+\Z_p^2).$$
Since $\DD_i+\Z_p^2$ is an open subgroup and $\Z[1/p]$ is dense, we see that $E_i$ is an extension of $\Z^2=\Z_p^2\cap\Z[1/p]^2$ by $(\DD_i+\Z_p^2)/\Z_p^2\simeq\Z(p^\infty)$, and $\Z[1/p]^2/E_i$ is isomorphic to $\Q_p^2/(\DD_i+\Z_p^2)\simeq\Z(p^\infty)$.

Set $e_{ij}^a = \id_n + a E_{ij}$ for $a \in \Z[1/p]$, $1 \le i \neq j \le n$.
Then the centre $Z$ of $A$ consists of the products $e_{14}^a \cdot e_{15}^b$.
The map
\begin{equation}
\label{eq:Definition-mu-tilde}
\mu \colon \Z[1/p]^2 \iso Z, \quad (a,b) \mapsto e_{14}^a \cdot e_{15}^b.
\end{equation}
is an isomorphism and shows that $Z$ is isomorphic to $\Z[1/p]^2$.
Define $V_i=\mu(E_i)$.
Note that $Z_0\subset V_i\subset Z$ and both $Z/V_i$ and $V_i/Z_0$ are isomorphic to $\Z(p^\infty)$. 
Notice that the group $C=\langle M_1\rangle$ normalizes $A$, its centre $Z$ as well as $Z_0$, $V_1$, and $V_2$. 
We are ready to define our group as

\begin{equation}
\label{eq:Definition-group-CB-rank-1}
B=(A/V_1) \rtimes C.
\end{equation}

If $\bar{V_1}$ denotes the image of $V_1$ in $\bar{A}=A/Z_0$, we see that $B=(\bar{A}/\bar{V_1})\rtimes C$. Since $\bar{A}$ and hence $\bar{A} \rtimes C$ are finitely presentable,
the quotient groups of $\bar{A}\rtimes C$ form an open neighbourhood of $B$ in the space of marked groups.
In what follows we need a smaller neighbourhood; 
to define it, we consider the (finite) set $W_0$ of elements of order $p$ in $\bar{Z}-\bar{V_1}$ and $W=W_0\cup\{w\}$, where $w$ is a fixed nontrivial element of $\bar{V_2}-\bar{V_1}$.

\begin{equation}
\label{eq:Definition-VV}
\VV = \OO^{\bar{A} \rtimes C}_{W} =  \{ N \triangleleft \bar{A} \rtimes C \mid N\cap W=\emptyset \}.
\end{equation}
The set $\VV$ is actually far smaller than its definition would suggest:
\begin{theorem} 
\label{ThCBone}
The set $\VV$ consists only of $\bar{V_1}\simeq\Z(p^\infty)$ and of finite subgroups of $\bar{Z}$, and so $B$ is an infinitely presented, nilpotent-by-abelian group of Cantor-Bendixson rank one.  
\end{theorem}
\subsection{Proof of Theorem \ref{ThCBone}} 
\label{ssec:Proof-of-theorem-3.8}
The proof is based on two lemmata. 
In each of them one of the defining properties of the matrix $M_0$ plays a crucial r\^{o}le.
The first lemma exploits property (\ref{iti}) and reads:
\begin{lemma}
\label{normalabels2}
Let $N$ be a normal subgroup of $A$ normalized by $C$ and not contained in $Z$. Then $N$ contains a finite index subgroup of $Z$. If $N$ is in addition assumed to contain $Z_0$, then it contains $Z$. 
\end{lemma}

\begin{proof}
The proof is similar in many respects to that of Lemma \ref{LemNormAbels}, so we skip some details. We first claim that if $N$ is any normal subgroup in $A$ and $N \nsubseteq Z$, then $N$ contains a nontrivial $\Z[1/p]$-submodule of $Z\simeq\Z[1/p]^2$.

Arguing as in the proof of Lemma \ref{LemNormAbels}, it follows that $N$ contains an element in $[U(A),U(A)]-Z$, namely of the form

\[
\begin{pmatrix}
1 & 0 & u_{13} & u_{14} & u_{15} \\
0 & 1 & 0 & u_{24} & u_{25} \\
0 & 0 & 1 & 0 & 0 \\
0 & 0 & 0 & 1 & 0 \\
0 & 0 & 0 & 0 & 1 \\
\end{pmatrix} \quad \text{with} \quad (u_{13},u_{24},u_{25})\neq(0,0,0).
\]
If $u_{13}\neq 0$, then by taking commutators with elementary matrices in the places $(3,4)$ and $(3,5)$ 
we obtain a subgroup of $Z$ of finite index.
If $(u_{24}, u_{25})\neq (0,0)$, 
we take commutators with elementary matrices in the $(1,2)$-place 
and 
find that $N$ contains the image under $\mu$ 
of the $\Z[1/p]$-submodule generated by $(u_{24},u_{25})$, so the claim is proved. 

If $N\nsubseteq Z$ is assumed in addition to be normalized by $C$, 
which acts $\Q$-irreducibly by hypothesis (\ref{iti}), 
we deduce that $N\cap Z$ has finite index in $Z$ in all cases. If $Z_0\subset Z$, it follows that the image of $N\cap Z$ in $Z/Z_0$ has finite index; since $Z/Z_0$ has no proper subgroup of finite index, it follows that $Z\subset N$.
\end{proof}

The second auxiliary result makes use of hypothesis (\ref{itii}).
\begin{lemma} \label{LemEigenLines}
Let $H$ be an $M_0$-invariant subgroup of $\Z[1/p]^2$. Suppose that $H$ contains $\Z^2$ and $H/\Z^2$ is infinite.
Then $H=\Z[1/p]^2$, or $H=p^{-n}\Z^2+E_i$ for some $i\in\{1,2\}$ and $n\ge 0$.
\end{lemma}

\begin{proof} We first claim that $H$ contains either $E_1$ or $E_2$. Let $H'$ be the closure of $H$ in $\Q_p^2$. The assumption $H/\Z^2$ infinite implies that $H'$ is not compact. Let us first check that $H'$ contains either $\mathcal{D}_1$ or $\mathcal{D}_2$. We begin by recalling the elementary fact that any closed, $p$-divisible subgroup of $\Q_p$ is a $\Q_p$-linear subspace: indeed any closed subgroup is a $\Z_p$-submodule, which $p$-divisibility forces to be a $\Q_p$-submodule.
Since $H'$ is not compact, for any $n\ge 0$, $p^nH'$ has nontrivial intersection with the 1-sphere in $\Q_p^2$. Therefore, by compactness, $\bigcap_{n\ge 0}p^nH'$ is a nonzero $M_0$-invariant and $p$-divisible closed subgroup and is therefore a $\Q_p$-subspace of $\Q_p^2$, thus contains some $\mathcal{D}_i$.

Thus $\mathcal{D}_i+\Z_p\subset H'$. Let us check that $E_i\subset H$. Since $E_i$ is contained in the closure $H'$ of $H$, for all $x\in E_i$ there exists a sequence $x_n\in H$ such that $x_n$ tends to $x$. But $x_n-x\in\Z[1/p]^2$ and tends to zero in $\Q_p^2$, so eventually belongs to $\Z^2$, which is contained in $H$. So for large $n$, $x_n-x$ and hence $x$ belongs to $H$ and $E_i\subset H$, so the claim is proved.

If $H\neq\Z[1/p]^2$, consider the largest $n$ such that $L=p^{-n}\Z^2\subset H$. We wish to conclude that $H=L+E_i$. The inclusion $\supset$ is already granted. 
Observe that $H/L$ contains a unique subgroup of order $p$. But an immediate verification shows that any subgroup of $\Z(p^\infty)^2$ strictly containing a subgroup isomorphic to $\Z(p^\infty)$, has to contain all the $p$-torsion. If $H$ contains $L+E_i$ as a proper subgroup, we apply this to the inclusion of $(L+E_i)/L\simeq\Z(p^\infty)$ into $H/L$ inside $\Z[1/p]^2/L\simeq\Z(p^\infty)^2$ and deduce that $H/L$ contains the whole $p$-torsion of $\Z[1/p]^2/L$, so that $H$ contains $p^{-(n+1)}\Z^2$, contradicting the maximality of $n$. So $H=L+E_i=p^{-n}\Z^2+E_i$ and we are done.
\end{proof}

\begin{proof}[Conclusion of the proof of Theorems \ref{ThCBone} and \ref{thm:Infinitely-related-CB-rank-1}]
Clearly, as a quotient of the finitely generated group $\bar{A}\rtimes C$ by $\bar{V_1}\simeq\Z(p^\infty)$, which is the increasing union of its finite characteristic subgroups and therefore is not finitely generated as a normal subgroup, the group $B$ is not finitely presentable.

Let $N$ be a normal subgroup of $\bar{A} \rtimes C$ that lies in the neighbourhood $\VV$ 
defined by equation \eqref{eq:Definition-VV}.  
Then Lemma \ref{normalabels2} (applied to the inverse image of $N\cap \bar{A}$ in $A\rtimes C$) implies that $N\cap\bar{A}\subset \bar{Z}$. Since $\bar{A}/\bar{Z}=A/Z$ contains its own centralizer in $A/Z\rtimes C$, it follows that $N\subset\bar{Z}$.

Assume now that $N$ is infinite. By Lemma \ref{LemEigenLines}, we have $H=\bar{Z}$, or $H=(p^{-n}Z_0+V_i)/Z_0$ for some $n\ge 0$ and $i\in\{1,2\}$. But all these groups have nonempty intersection with $W$, except $\bar{V_1}$ itself. Thus the only infinite subgroup $N$ in $\mathcal{V}$ is $\bar{V_1}$.

As finite subgroups of $\bar{Z}$ are isolated in the space of subgroups of $\bar{Z}\simeq\Z(p^\infty)^2$ by an easy argument (see for instance \cite[Prop.~2.1.1]{CGP10}), it follows that $\VV-\{\bar{V_1}\}$ consists of isolated points; clearly $\bar{V_1}$ is not isolated as it can be described as the increasing union of its characteristic subgroups and it follows that $\bar{V_1}$ has Cantor-Bendixson rank one in $\VV$, and therefore $\bar{V_1}$ has Cantor-Bendixson rank one in the space of normal subgroups of the finitely presentable group $\bar{A}\rtimes C$, so $B=(\bar{A}\rtimes C)/\bar{V_1}$ has Cantor-Bendixson rank one.

To conclude the assertion in Theorem \ref{thm:Infinitely-related-CB-rank-1} that $B$ is nilpotent-by-abelian, observe that the obvious decomposition $A=U(A)\rtimes D$ (where $D$ is the set of diagonal matrices in $A$) extends to a semidirect decomposition $A\rtimes C=U(A)\rtimes (D\times C)$, which is nilpotent-by-abelian and admits $B$ as a quotient.
\end{proof}

\begin{remarks}
\label{remarks:Comment-proof-thm-3.7}
(a) The trivial subgroup of $B$ is isolated in $\NN(B)$ (i.e., $B$ is finitely discriminable, see \S \ref{parat}), for the above proof shows that any nontrivial normal subgroup of $B$ contains the unique subgroup of order $p$ in $Z/V_1$.

(b) If we require that modulo $p$, the matrix $M_0$ has two distinct eigenvalues, then Theorem \ref{ThCBone} can be slightly improved: first it is enough to pick $W=\{w\}$, where $w$ is an element of order $p$ in $\bar{V_2}$, and the set $\VV$ consists only of $\bar{V_1}\simeq\Z(p^\infty)$ and its finite subgroups. Note that without this further assumption, it can happen that the intersection $\bar{V_1}\cap\bar{V_2}$ be nontrivial, and even if this intersection is trivial, the finite subgroups in $\mathcal{V}$ are not necessarily all contained in $\bar{V_1}$.
\end{remarks}

\section{Independent families and presentations}\label{secip}
This section centers around the concepts of an independent family of normal subgroups 
and of an infinitely independently presentable group (INIP group); 
they have been introduced in Definitions \ref{definition:independent-family} and \ref{definition:INIP}.
We establish, in particular, a criterion for INIP groups based on the Schur multiplier.
\smallskip

\emph{Notation. } Let $\{N_i \mid i \ I \}$ be a family of normal subgroups of a group $G$.
For every subset $J \subseteq I$ we denote by $N_J$ the normal subgroup of $G$ generated by the normal subgroups $N_j$ with $j \in J$. 

\subsection{Infinitely independently presentable groups}\label{INIP-groups} 
We begin with a remark. 
A well-known result of B. H. Neumann states 
that whether or not a finitely generated group can be defined by a finite set of relations 
does not depend on the choice of the finite generating system (\cite[Cor.\,12]{Neu37}).
Infinitely independently presented groups enjoy a similar property:
\begin{lemma}\label{indpre}
Let $G$ be a finitely generated group 
and $\pi \colon F \epi G$ an epimorphism of a free group $F$ of finite rank onto $G$.
If $R = \ker \pi$ is generated 
by the union of an infinite independent family of normal subgroups $R_i \triangleleft F$,
the kernel of every projection $\pi_1 \colon F_1 \epi G$ with $F_1$ a free group of finite rank
has this property.
\end{lemma}

\begin{proof}
It suffices to establish, for every epimorphism  $\rho \colon \tilde{F} \epi F$ of a free group $\tilde{F}$ of finite rank, 
that the kernel $R = \ker \pi$ is generated by an infinite independent family of normal subgroups of $F$ 
if, and only if, $\tilde{R} = \ker ( \rho \circ \pi)$ has this property.
Assume first that $R$ is generated by a family $\{R_i \mid i \in I\}$. 
Then $\tilde{R}$ is generated by the family of normal subgroups $\rho^{-1}(R_i)$ 
and this family is independent.
Conversely, assume $\tilde{R}$ is generated by an infinite independent family $\{\tilde{R}_i \mid i \in I \}$.
The kernel of $\rho$ is finitely generated as a normal subgroup of $\tilde{F}$ and contained in $\tilde{R}$.
So there exists a finite $J \subset I$ so that $\ker \rho \subseteq \tilde{R}_J$.
For $i \in I \smallsetminus J$, set $R_i  = \rho(\tilde{R}_i \cdot \tilde{R}_J) $.
Then $\{R_i \mid i \in I \smallsetminus J \}$ 
is an infinite independent family of normal subgroups generating $R$.
\end{proof}

\subsubsection{Schur multiplier}\label{schur}
The next result exhibits a first connection between infinite independent families of subgroups 
and infinitely independently presented groups.
\begin{proposition}\label{p_schur}
Let $G$ be a finitely generated group 
and assume that the Schur multiplier $H_2(G,\Z)$ is generated by an infinite independent family of subgroups.
Then $G$ is infinitely independently presentable.
\end{proposition}
\begin{proof}
Let $F$ be a finitely generated free group with an epimorphism $\pi\colon  F\epi G$; 
let $R$ denote its kernel. 
The extension $R\mono F\epi G$ gives rise to an exact sequence in homology
$$
H_2(F,\Z)\to H_2(G,\Z)\to R/[R,F]\to F_{\ab}\stackrel{\pi_{\ab}}\epi G_{\ab}
$$
(see for instance \cite[Cor.~8.2]{HiSt97}). 
Since $F$ is free, its multiplier is trivial, so the above exact sequence leads to the extension
$$
H_2(G,\Z)\mono R/[R,F]\epi \ker \pi_{\ab}.
$$
Since $F_{\ab}$, and hence $\ker \pi_{\ab}$, are free abelian, this sequence splits; 
say $R/[R,F] = H_2(G, \Z) \oplus B$.
By hypothesis, the direct summand $H_2(G,\Z)$ is generated by an infinite, independent family of subgroups, 
say $\{A_i \mid i \in I\}$. 
Then $i \mapsto A_i \oplus B$ is an infinite independent family of subgroups generating $R/|R,F]$.
Its preimages under the canonical map $R \epi R/[R,F]$ form then a family 
which reveals that $G = F/R$ is INIP.
\end{proof}

\begin{remark} \label{Multiplicator-fg-generated}
The multiplicator of a countable group, in particular of a finitely generated group, is countable.
The index set $I$ of the family $\{A_i \mid i \in I\}$ occurring in the above proof  
can therefore be taking to be $\N$.
\end{remark}

\begin{example}
Examples of finitely generated groups satisfying the assumptions of Proposition \ref{p_schur} (and thus INIP) are groups of the form $G=F/[R,R]$ with $F$ is free and $R$ is an arbitrary normal subgroup of infinite index. 
Indeed, it is shown in \cite{BST80} 
that $H_2(G,\Z)$ has then a free abelian quotient of infinite rank.
\end{example}

\begin{example}
The first Grigorchuk group $\Gamma$ is an example of a group with growth strictly between polynomial and exponential \cite{Gri84}.
In \cite{Gri99}, Grigorchuk proves that $\Gamma$ has an infinite minimal presentation, and more precisely that the ``Lys\"enok presentation" is minimal, by showing that it projects to a basis of the Schur multiplier $H_2(G,\Z)$, which is an elementary 2-group of infinite rank.
\end{example}

\subsubsection{Characterization of abelian groups with an infinite independent family of subgroups}
\label{Characterization-abelian-groups-with-INI-family}
In the statement of Proposition \ref{p_schur}, 
the multiplier is assumed to be generated by an infinite independent family of subgroups.
This requirement admits several reformulations.
We first recall the notion of an (abelian) minimax group: 
an abelian group is called \emph{minimax} if it has a finitely generated subgroup $B$ 
such that $A/B$ is artinian, i.e.\ satisfies the descending chain condition on subgroups
(see \cite{Rob96}, Exercice 4.4.7 and Theorem 4.2.11 for more details).%

\begin{lemma}\label{abq}
Let $A$ be an abelian group. The following are equivalent
\begin{enumerate}[(i)]
\item\label{aa1} $A$ is not minimax;
\item\label{aaq} the poset $(\mathcal{N}(A),\subset)$ contains a subposet isomorphic to $(\Q,\le)$;
\item\label{aap} the poset $(\mathcal{N}(A),\subset)$ 
contains a subposet isomorphic to the power set $(2^\Z,\subset)$;
\item\label{aa2} $A$ has an infinite independent sequence of (normal) subgroups;
\item\label{aa3} $A$ has a quotient that is an infinite direct sum of nonzero groups.
\end{enumerate}
\end{lemma}
\begin{proof}
It is clear that \eqref{aa3} implies \eqref{aa2} 
and biimplication \eqref{aa2} $\Leftrightarrow$ \eqref{aap} is immediate
from Definition \ref{definition:independent-family}.
Now to implication \eqref{aap} $\Rightarrow$ \eqref{aaq}.
The poset $(2^\Z,\subset)$ is isomorphic to $(2^\Q,\subset)$ 
(an isomorphism being induced by any bijection of $\Z$ onto $\Q$). 
The power set $2^\Q$ obviously contains the set of intervals $\mathopen]-\infty,s]$ 
with $s$ ranging over rationals. It follows that  $(\Q,\le)$ is a subposet of $(2^\Q,\subset)$,
and hence of the poset  $(2^\Z,\subset)$.

We continue with implication \eqref{aaq} $\Rightarrow$ \eqref{aa1}.
Suppose $A$ is an abelian group so that the poset $(\mathcal{N}(A),\subset)$ 
contains an infinite,  linearly and densely ordered subset.
Then $A$ does not satisfy the descending chain condition 
and neither does a quotient group $A/B$ of $A$ modulo a finitely generated subgroup $B$.
So $A$ is not minimax.
\smallskip

We are left with proving that (\ref{aa1}) implies (\ref{aa3}). 
We first claim that $A$ has a quotient that is torsion and not minimax, hence not artinian. 
Indeed, let $(e_i)_{i\in I}$ be a maximal $\Z$-free family in $A$, generating a free abelian group $B$. Then $A/2B$ is torsion and is not minimax: 
indeed if $I$ is finite, $2B$ is minimax. but not $A$, and so $A/2B$ is not  minimax
(the property of being minimax is closed under extensions); 
if $I$ is infinite, $A/2B$ admits $B/2B\simeq(\Z/2\Z)^{(I)}$ as a subgroup so is not minimax.

Replacing, if need be,  $A$ by the quotient $A/2B$, 
we can thus assume that $A$ is torsion. 
If $A$ admits $p$-torsion for infinitely many $p$'s, 
it is itself an infinite direct sum (of its $p$-primary components) and we are done. 
If, on the other hand,
$A$ has $p$-torsion for only finitely many primes $p$, 
one of the $p$-primary component will not be minimax;
by replacing $A$ by this $p$-primary component,
we can suppose that $A$ is a $p$-group for some prime $p$.

If $A/pA$ is infinite, it an infinite direct sum of copies of $\Z/p\Z$ and we are done. 
Otherwise, let $F$ be the subgroup generated by a transversal of $pA$ in $A$; 
$F$ is finite and $A/F$ is divisible and non-minimax, hence isomorphic to $\Z(p^\infty)^{(I)}$ for some infinite $I$. 
\end{proof}

Note that the proof readily shows that every abelian non-minimax group $A$ has a quotient 
which is isomorphic to either $(\Z/p\Z)^{(I)}$, $\Z(p^\infty)^{(I)}$ for some infinite set $I$ and some prime $p$, 
or is a direct sum $\bigoplus_{p\in P}\Z/p\Z$ or $\bigoplus_{p\in P}\Z(p^\infty)$ for some infinite set $P$ of primes.
Note also that an abelian group that is not minimax
can be very far from being or containing a nontrivial infinite direct sum, e.g.\ $\Q$ is such an example.
\subsection{Groups with an infinite independent family of normal subgroups}
\label{Groups-with-independent-family}
There are many types of finitely generated groups with an infinite independent family of normal subgroups.
Here we restrict ourselves to two particular classes of such groups; 
both of them consist actually of  intrinsic condensation groups (see Lemma \ref{indcon}).

\subsubsection{Normal subgroups which are infinite direct products}
\label{Normal-but-infinite-direct-sums}
Let $G$ be a group which has a normal subgroup 
that can be written as as infinite (restricted) direct product  $\bigoplus_{i \in I} H_i$, 
each $H_n$ being a nontrivial normal subgroup of $G$. 
The family $\{H_i \mid i \in I\}$  is clearly independent;
it has the additional property that for every finite subset $\FF' \subset G$ 
there exists a finite subset $J \subset I$ such $\{H_i \mid i \in I \smallsetminus J\}$ 
is an infinite independent family which generates a normal subgroup that is disjoint from $\FF'$.

Here are some examples of such groups:
\begin{itemize}
\item Any group $G$ whose centre $Z$ contains an infinite direct sum. 
Finitely presented examples (which are in addition solvable) appear in \cite[Sec.~2.4]{BGS86}.

\item The two generator groups introduced by B.\,H.\,Neumann in \cite[Thm.~14]{Neu37}.
Each of them contains an infinite direct product  of finite alternating groups $\bigoplus H_n$ as normal subgroups
(with each $H_n$ normal in $G$). 

\item Any group which is isomorphic to proper direct factor of itself, 
i.\,e., $G$ isomorphic to $G\times H$ where $H$ is a nontrivial group: indeed, if $\phi$ is an injective endomorphism of $H\times G$ with image $\{1\}\times G$, then we can set $H_n=\phi^n(H\times\{1\})$, and clearly $H_n$ is normal and the $H_n$ generate their direct sum. Finitely generated examples of such groups (actually satisfying the stronger requirement $G\simeq G\times G$) are constructed in \cite{Jon74}. Note that a group isomorphic to a proper direct factor is in particular isomorphic to a proper quotient, i.e.\ is {\it non-Hopfian}. However, a finitely generated non-Hopfian group does not necessarily contain a normal subgroup that is an infinite direct sum of non-trivial normal subgroups (anticipating again on Section \ref{sec:con}, it is not even necessarily of intrinsic condensation). Indeed, if $n\ge 3$ and $A_n$ is the Abels' group introduced in Section \ref{sec:Group-with-CB-rank-1} and $Z_0$ an infinite cyclic subgroup of its centre, then $A_n/Z_0$ is not Hopfian 
(the argument is due to P. Hall for $A_3$ and extends to all $n$ \cite{Hal61,Abe79}), 
but Lyulko's result (Proposition \ref{CorCountQuot}) implies that $A_n$ has only countably many subgroups.
\end{itemize}

\subsubsection{Independent varieties}\label{aovl}
The following result is due to Adian, Olshanskii and Vaughan-Lee independently
(see \cite{Adj70,Ols70,VaL70}). 
It will be used in a crucial way in the proof of Theorem \ref{thm:Non-contraction}.

\begin{theorem}[Adian, Olshanskii, Vaughan-Lee]\label{ols}
If $F$ is a free group over at least two generators (possibly of infinite rank), then $F$ has uncountably many fully characteristic subgroups (i.e.\ stable under all endomorphisms). 
Actually, it has a countably infinite  independent family of fully characteristic subgroups. 
\end{theorem}

\begin{corollary}\label{ontofree}
If $H$ is a group admitting an epimorphism onto a non-abelian free group $F$, then $H$ has at least continuously many fully characteristic subgroups. 
\end{corollary}
\begin{proof}
Our aim is to construct, given a fully characteristic subgroup $N$ of $F$,
a fully characteristic subgroup $\Phi(N)$ of $H$ in such a way that the function $N \mapsto \Phi(N)$ is injective. Theorem \ref{ols} will then imply the conclusion of Corollary \ref{ontofree}.

We begin by fixing the notation.
For a group $G$, let $\mathcal{FCH}(G)$ denote the set of its fully characteristic groups.
Next, given a group $L$,
let $U_G(L)$ denote the intersection of the kernels of all homomorphisms $g \colon G\to L$ from $G$ into $L$; 
in symbols
\[
U_G(L) = \bigcap \;\{\ker g \mid g \colon G \to L \}.
\]
It is easily checked that $U_G(L)$ is a fully characteristic subgroup of $G$.

We apply this construction twice: 
first with $G = H$ and $L = F/N$ where $N \in \mathcal{FCH}(F)$
and obtain the group $\Phi(N) \in \mathcal{FCH}(H)$; in symbols,
\begin{equation}\label{eq:Def-Phi}
\Phi(N) = U_H(F/N) = \bigcap \;\{\ker h \mid h \colon H \to F/N \}.
\end{equation}
Secondly, we apply it with $G =F$ and $L = H/M$ where $M \in \mathcal{FCH}(H)$,
obtaining the fully characteristic subgroup
\begin{equation} \label{eq:Def-Psi}
\Psi(M) = U_F(H/M) = \bigcap \;\{\ker f \mid f \colon F \to H/M \}
\end{equation}
of $F$. We claim that $\Psi \circ \Phi$ is the identity of $\mathcal{FCH}(F)$.

Given $N \in \mathcal{FCH}(F)$, set $M = \Phi(N)$ and $N' = \Psi(M)$.
We first show that $N' \subseteq N$. 
Since $F$ is free, the projection $\pi \colon  H \epi F$ admits a section $\iota \colon F \mono H$.
By \eqref{eq:Def-Psi}, 
the group $N' = \Psi(M)$ is contained in the kernel of the homomorphism
$f_0 = \can_M \circ \iota \colon F \mono H \epi H/M$. 
Next, by \eqref{eq:Def-Phi},
the group  $M$ is contained in the kernel of the map
$h_0 = \can_N \circ \pi \colon H \epi F \epi F/N$.
Since $\ker h_0 = \pi^{-1}(N)$ it follows that
\[
N'  \subseteq \ker f_0 \subseteq \ker (F \mono H \epi H/(\pi^{-1}(N) ) 
= \iota^{-1}\left(\pi^{-1} (N)\right) 
= (\pi \circ \iota)^{-1}(N) = N.
\]
The opposite inclusion $N \subseteq N'$ will be proved by showing 
that $F \smallsetminus N' \subseteq  F \smallsetminus N$.
Suppose that $x \in F \smallsetminus N' = F \smallsetminus  \Psi(M)$. 
Then there exists a homomorphism 
$f \colon  F  \to H/M$ such that $f(x) \neq 1$.
Since $F$ is free, $f$ can be lifted to a homomorphism $\tilde{f} \colon F \to H$ with  $\tilde{f} (x) \notin M$.
By the definition of $M = \Phi(N)$, there exists a homomorphism $h \colon H \to F/N$ such that 
$h ( \tilde{f}(x)) \neq 1 \in F/N$. 
The composition $\varphi = h \circ \tilde{f} \colon F \to F/N$ then  lifts to an endomorphism 
$\tilde{\varphi}$ of $F$ with $\tilde{\varphi}(x) \notin N$.
But $N$ is fully characteristic, so it is mapped into itself by $\tilde{\varphi}$ and thus $x \in F \smallsetminus N$.
\end{proof}

\subsection{Finitely generated groups with no minimal presentation.}
\label{Fg-no-minimal-presentation}
The notion of a minimal presentation has been mentioned at the beginning of Section \ref{intro}.
Here is a formal definition:
\begin{definition}\label{minimal-presentation}
We say that a group presentation $\langle S\mid \mathcal{R}\rangle$ is \emph{minimal} 
if $S$ is finite and no relator $r \in \mathcal{R}$ is a consequence of the remaining relators.
If a group (finitely generated) $G$ admits such a presentation, it is called \emph{minimally presentable}.
\end{definition}

Definition \ref{minimal-presentation} can be rephrased like this:
let $S$ be a finite set, $F_S$ the free group on $S$ and $\mathcal{R}$ a subset of $ F_S $.
The presentation $\langle S\mid \mathcal{R}\rangle$ is minimal if and only if 
the family of normal subgroups  $\{N_r \mid r \in \mathcal{R} \}$ generated 
by the individual relators $r$ is independent.

\begin{example}
\label{example:minimal}
In \cite{Neu37} B. H. Neumann published the first, explicit examples of infinitely presented groups.
He considers the free group on $a$ and $b$, the commutator words
\[
c_n = [a^{-n}ba^{n}, b]
\]
and shows that for $I = \{2^k-1 \mid k \geq 1 \}$ 
no relator $c_i$ with $i \in I$ is a consequence of the commutators $c_j$ with $j \in I \smallsetminus \{i\}$
(see \cite[Thm.\;13]{Neu37}).
If follows that 
\[
\langle a, b \mid [a^{1 -2^k }ba^{2^k - 1}, b] = 1 \text{ for } k \geq 1 \rangle
\] 
is a minimal presentation of the group $G$ defined by the presentation; 
in particular, $G$ is infinitely presented.
The group defined by the presentation 
\[
\langle a, b \mid [a^{-n}ba^{n}, b] = 1 \text{ for } n \geq 1 \rangle
\]
is  the standard wreath product $\Z \wr C_\infty$ of two infinite infinite cyclic groups.
In \cite{Bau61}, 
G.\ Baumslag proves that this presentation is minimal (see, \cite{Str84} for a shorter proof).
\end{example}

Every finitely presentable group is minimally presentable, 
since we can remove successively redundant relators. 
If an infinitely presented group admits a minimal presentation, 
the family of relators has to be infinite countable 
and the group is INIP in the sense of definition \ref{definition:INIP}.

\subsubsection{Groups admitting no minimal presentation}
\label{No-minimal-presentation}
An infinitely presented group need not admit a minimal presentation.
Examples are provided by the central quotients $A_n/Z$ of the groups $A_n$, 
considered in section \ref{ssec:Abels-groups},
in view of the facts that $A_n$ is finitely presentable for $n \geq 4$,
the centre of $A_n$ is isomorphic to $\Z[1/p]$, its quotient $ Z = \Z[1/p]/\Z$ is divisible, 
and the following result:
\begin{proposition}\label{nominp}
Let $G$ be a finitely presentable group and $Z$ a central subgroup. Assume that $Z$ is divisible and non-trivial. Then $G/Z$ is not minimally presentable.
\end{proposition}
\begin{proof}
Since $Z$ is divisible and non-trivial, it is not finitely generated and therefore $G/Z$ is infinitely presented.
Let $F$ be a free group of finite rank with an epimorphism $F\epi G$ and let $K$ be its kernel. 
Let  $(u_n)_{n\ge 0}$ be a set of relators for $G/Z$ over $F$. 
Since $G/Z$ is infinitely presented, the set $(u_n)$ is infinite. 
Since $K$ is finitely generated as a normal subgroup, 
it is contained in the normal subgroup generated by $u_0$, \ldots, $u_k$ for some $k$. 
Let $M$ be the normal subgroup of $F$ generated by all $u_i$ for $i\neq k+1$. 
Then $M$ contains $K$, and thus $M/K$ can be identified with a subgroup $Z'$ of $Z$.
 The group $Z/Z'$ is generated by the image of $u_{k+1}$;
 being a quotient of $Z$, it is divisible and cyclic, hence trivial. 
 So $Z'=Z$, i\,.e.,  $M$ is the kernel of $F\epi G/Z$. 
 Therefore the set of relators $(r_n)$ is not minimal.
\end{proof}

\subsubsection{Infinitely independently presented groups with no minimal presentation}
\label{INIP-not-minimal}
The notion of infinitely minimally presentable groups is at first sight more natural than the notion of INIP groups; however, it is much less convenient to deal with. Consider, for instance, a finitely generated group $\Gamma$ that decomposes as a direct product $\Gamma=\Gamma_1\times\Gamma_2$. 
It is immediate that if $\Gamma_1$ is INIP, then so is $\Gamma$. 
Is it true that if $\Gamma_1$ is infinitely minimally presentable, then so is $\Gamma$?

The following result allows one to construct infinitely independently presentable groups 
which admit no minimal presentation.
\begin{corollary}
\label{cor:INIP-not-minimal}
Let $G$ be a finitely presentable group and $Z$ a central subgroup. 
Assume that $Z$ is divisible but not minimax. 
Then $G/Z$ is infinitely independently presentable, but not minimally presentable.
\end{corollary}
\begin{proof}
Combine Propositions \ref{p_schur} and \ref{nominp}.
\end{proof}
\begin{example}\label{e_oh}
In \cite{OuH07},
Ould Houcine proves that every countable abelian group embeds into the centre of a finitely presentable group. 
So there exist groups to which Corollary \ref{cor:INIP-not-minimal} applies. 
\end{example}

\section{Condensation groups}\label{sec:con}
 
\subsection{Generalities on condensation groups}

Recall that if $G$ is a group, we denote by $\mathcal{N}(G)$ the set of its normal subgroups; this is a closed (hence compact) subset of $2^G$.

We say that a finitely generated group $G$ is
\begin{itemize}
\item of {\em condensation} if given some (and hence any) marking of $G$ by $m$ generators, $G$ lies in the condensation part of $\mathcal{G}_m$;
\item of {\em extrinsic condensation} if for every finitely presentable group $H$ and any epimorphism $H\epi G$ with kernel $N$, there exist uncountably many normal subgroups of $H$ contained in $N$; 
\item of {\em intrinsic condensation} if $\{1\}$ lies in the condensation part of $\mathcal{N}(G)$ (i.e.\ $G$ lies in the condensation part of the set of its own quotients, suitably identified with $\mathcal{N}(G)$). The notion of intrinsic condensation was introduced in \cite[Sec.~6]{CGP07}.
\end{itemize}

Note that being of intrinsic condensation makes sense even if $G$ is not finitely generated. There is the following elementary lemma.

\begin{lemma}\label{condei}
For every finitely generated group $G$, the following statements hold: \begin{enumerate}[(1)]
\item\label{ec} If $G$ is of extrinsic condensation, then it is condensation and infinitely presented;
\item\label{ic} if $G$ is of intrinsic condensation, then it is condensation; 
if,moreover, $G$ is finitely presentable, the converse is true.
\end{enumerate}
\end{lemma}
\begin{proof}
(\ref{ec}) It is clear from the definition that extrinsic condensation implies infinitely presented. To check that it also implies condensation, assume $G \in \GG_m$ is of extrinsic condensation 
 and $N$ is the kernel of the canonical epimorphism $F_m \epi G$. 
 Consider an open neighbourhood $\OO_{\FF, \FF'}$ of $N$ in $\GG_m$, 
 with $\FF$ and $\FF'$ finite subsets of $F_m$ as in (\ref{eq:Basis-topology}).
Define $N_{\FF} \triangleleft F_m$ to be the normal subgroup generated by $\FF$.
Then $H = F_m/N_{\FF}$ is a finitely presentable group mapping onto $G$;
since $G$ is of extrinsic condensation, 
the kernel $N/N_{\FF}$ of $H \epi G$, and hence the neighbourhood $\OO_{\FF, \FF'}$,
contain uncountably many normal subgroups.
It follows that $G$ is a condensation group.

(\ref{ic}) Since given a marking of $G$ there is a homeomorphic embedding of $\mathcal{N}(G)$ into $\mathcal{G}_m$ mapping $\{1\}$ to $G$, the first implication is obvious. The converse follows from the fact that if $G\in\mathcal{G}_m$ is finitely presentable, 
then some neighbourhood of $G$ consists of quotients of $G$.
\end{proof}

Condensation and independency are related by the following facts:
\begin{lemma}\label{indcon}
Let $G$ be group and $\{M_i \mid i\in I\}$ an independent family of normal subgroups. Then
\begin{enumerate}[(1)]
\item\label{ind1} if $G$ is finitely generated and $I$ is infinite then $G/M_I$ is of extrinsic condensation. Thus INIP implies extrinsic condensation.
\item\label{ind2} if $I$ is infinite and the family $\{M_i \mid i\in I\}$ generates the (restricted) direct product of the $M_i$, 
then $G$ is of intrinsic condensation.
\end{enumerate}
\end{lemma}
\begin{proof}
We begin with (\ref{ind2}). 
Let $\mathcal{F}'$ be a finite subset of $G \smallsetminus \{1\}$.
There exists then a finite subset $J \subset I$ 
such that $\mathcal{F}' \cap M_I \subseteq \mathcal{F}' \cap M_J$.
Since the family $\{M_i \mid i\in I\}$ generates its direct sum, 
$M_{I \smallsetminus J}$ is disjoint from $\mathcal{F}'$  and contains $2^{\aleph_0}$ normal subgroups.
This shows that $G$ is an intrinsic condensation group.

Let us now prove (\ref{ind1}). Let $H$ be a finitely presentable group with an epimorphism $\pi:H\epi G$. Let $H_1$ be a finitely presentable group with an epimorphism $\rho:H_1\epi G$. Then the $\rho^{-1}(M_i)$ are also independent, so $G/M_I$ is INIP relative to $H_1$. By Lemma \ref{indpre}, it follows $G/M_I$ is also INIP relative to $H$, say by an infinite independent family of normal subgroups $(N_j)_{j\in J}$. The subgroups $N_K$, when $K$ ranges over subsets of $J$, are pairwise distinct and contained in the kernel of $\pi$, and there are uncountably many such subgroups. By definition, this shows that $G/M_I$ is of extrinsic condensation.
\end{proof}

\begin{remark}\label{rem:INIP-implies-extrinsic-condensation}
We therefore have
\begin{align*} \text{Infinitely minimally presentable}\;\Rightarrow\;\text{INIP}\;\Rightarrow\;\\ \text{Extrinsic Condensation}\;\Rightarrow \;\text{Infinitely Presented}.\end{align*}
The first implication is strict by Example \ref{e_oh}. Examples \ref{abelsnik} and \ref{abelsni}, based on Abels' group, show that the two other implications are strict.
\end{remark}

A useful corollary of Theorem \ref{ols} is the following.

\begin{corollary}\label{corols}
Every group $G$ with a non-abelian free normal subgroup has an infinite independent sequence of normal subgroups and is of intrinsic condensation.
\end{corollary}
\begin{proof}
Let $F$ be a non-abelian free normal subgroup in $G$. By Theorem \ref{ols}, 
there is in $F$ a infinite independent sequence $ n \mapsto H_n$ of characteristic subgroups; they are therefore normal in $G$ and independent. In particular, $F$ contains uncountably many normal subgroups of $G$. By Levi's theorem (see, e.\;g., \cite[Chap.~I, Prop.~3.3]{LS77})
the derived subgroups $F^{(n)}$ (which are also free and non-abelian) intersect in the unit element, and therefore in $\mathcal{N}(G)$, every neighbourhood of $\{1\}$ contains the set of subgroups of $F^{(n)}$ for some $n$ and is therefore uncountable. 

If the free group $F$ is countable, the preceding argument shows that $\{1\}$ lies in the condensation part of $G$. The argument fails if $F$ is uncountable, since for a closed subset of $2^F$, to show that an element is condensation it is not enough to show that all its neighbourhoods are uncountable.
Let $F$ be uncountable and let $F'$ be a free factor of infinite countable rank, and $\mathcal{FCH}(F)$ denotes the set of fully characteristic subgroups of $F$, then the natural continuous map $c:\mathcal{FCH}(F)\to\mathcal{FCH}(F')$, defined by $c(N)=N\cap F'$, is injective.
Indeed, let $m$ be an element of $F$. Then there exists an automorphism of $F$ (induced by a permutation of generators) mapping $m$ to an element $m'\in F'$. This shows that for every characteristic subgroup $N$ of $F$, we have $N=\bigcup_\alpha\alpha(N\cap F')$, where $\alpha$ ranges over automorphisms of $F$, and thus $N$ is entirely determined by $N\cap F'=c(N)$. Thus by compactness, $c$ is a homeomorphism onto its image. Since $\{1\}$ is a condensation point in $\mathcal{FCH}(F')$, we deduce that $\{1\}$ is a condensation point in $\mathcal{FCH}(F)$.
\end{proof}

\subsection{Examples of intrinsic condensation groups}

\subsubsection{Free groups and generalizations}\label{sssfrg}
The free group $F_2$ is an example of an intrinsic condensation group; 
this is a well-known fact and follows, for instance, from Corollary \ref{corols}.
Here is a proof that does not rely on Theorem \ref{ols}.

Let $F_2$ be a free group of rank 2
and $\{w_i \mid i \in I\}$  an infinite subset of words in $F_2$.
Suppose the family of normal subgroups $i \mapsto \gp_{F_2}(w_i)$ is independent.
Fix an index $i_0$, set $J = I \smallsetminus \{i_0\}$ 
and define the new subset $\{w'_j = w_{i_0}^{-1}\cdot w_j \mid j \in J \}$.
It produces a family of normal subgroups $\gp_{F_2}(w'_j)$ that is independent, too.

Consider now a sequence of words $n \mapsto w_n$  be in $F_2$
for which the associated sequence of normal subgroups $ n \mapsto \gp_{F_2}(w_n)$ is independent,
for instance the sequence of commutator words
\[
w_n = [a^{-n} ba^{n}, b] \quad \text{with}  \quad n \in \N^* = \N \smallsetminus \{0\}
\]
discussed in  Example \ref{example:minimal}.
Given a finite finite subset $\FF' \subset F_2 \smallsetminus \{1\}$
one can find a normal subgroup $N \triangleleft F_2$ of finite index that avoids $\FF'$
(use that $F_2$ is residually finite; cf.\,\cite[6.1.9]{Rob96}).
There exists then a coset $f\cdot N$ of $N$ 
so that the subset $I = \{n  \mid w_n \in f\cdot N \}$ is infinite.
Fix $i_0 \in I$ and set $w_j' = w_{i_0}^{-1} w_j$ for $j \in J = I \smallsetminus \{i_0\}$.
By the previous paragraph 
the words $w_j'$ lead to an independent family of normal subgroups  $\{S_j \mid j \in J\}$;
as each $S_j$ is contained in $N$ the $S_j$ generate a normal subgroup that is disjoint from $\FF'$.
It follows that the neighbourhood 
\[
\OO^{F_2}_{\FF'} = \{ S\triangleleft F_2 \mid  S \cap \FF' =  \emptyset \}
\]
of $\{1\}$ in $\NN(F_2)$  contains $2^{\aleph_0}$ normal subgroups.
\smallskip

By Corollary \ref{corols}
every group with a non-abelian free normal subgroup is of intrinsic condensation.
Here are some classes of finitely generated groups to which the corollary applies:
\begin{itemize}
\item \emph{Non-abelian limits groups.} These are finitely generated groups that are limits of non-abelian free groups in the space of marked groups. The existence of a non-abelian free subgroup follows from \cite[Thm.~4.6]{ChGu05}.
\item \emph{Non-elementary hyperbolic groups.} 
The existence of a non-abelian normal free subgroup was established by Delzant \cite[Thm.~I]{Del96}; 
the intrinsic condensation can also be derived from \cite[Thm.~3]{Ols93}.
\item \emph{Baumslag-Solitar groups}
\[
\textnormal{BS}(m,n)=\langle t,x \mid tx^mt^{-1}=x^n\rangle\qquad (|m|,|n|\ge 2);
\]
indeed, the kernel of its natural homomorphism onto $\Z[1/mn]\rtimes_{m/n}\Z$ is free and non-abelian: 
it is free because its action on the Bass-Serre tree of $\textnormal{BS}(m,n)$ is free. It is non-abelian, because otherwise $\textnormal{BS}(m,n)$ would be solvable, but it is a non-ascending HNN-extension so contains non-abelian free subgroups. 
\item Many other HNN-extensions obtained by truncating some group presentations, 
see Theorem \ref{thm:Structure-theorem}.
\end{itemize}

\begin{problem}
Does there exist a non-ascending HNN-extension having no nontrivial normal free subgroup?
\end{problem}
A natural candidate to be a counterexample would be a suitable HNN-extension of a non-abelian free group over a suitable isomorphism between two maximal subgroups. Recall that in \cite{Cam53}, simple groups are obtained as amalgams of non-abelian free groups over maximal subgroups of infinite index.

\subsubsection{Intrinsic condensation and abelian groups}
By a result of D. L. Boyer, 
a countably infinite abelian group $A$ has either $\aleph_0$ or $2^{\aleph_0}$ subgroups;
the second case occurs if and only if $A$ is either not minimax 
or if it admits $\Z(p^\infty)^2$ as a quotient for some prime $p$ (\cite[Thm.\,1]{Boy56}). 
The stated characterization of abelian groups with $2^{\aleph_0}$ subgroups
characterizes also the abelian groups $A$
that are condensation points in $\mathcal{N}(A)$:
Boyer \cite{Boy56} from which the first equivalence above can be checked directly.

\begin{proposition}[\protect{\cite[Prop.~A and Thm.~G]{CGP10}}] 
\label{p_condsa}
Let $A$ be an abelian group. Then 
\begin{enumerate}[(1)]
\item\label{aec} $A$ is a condensation point in $\mathcal{N}(A)$ if and only if $A$ is either not minimax 
or admits $\Z(p^\infty)^2$ as a quotient for some prime $p$; 
\item\label{aic} $A$ is of intrinsic condensation (i.e., $\{0\}$ is a condensation point in $\mathcal{N}(A)$) 
if and only if either $A$ is minimax or there exists a prime $p$ 
such that $\Z(p^\infty)^2$ is a quotient of $A$ and  $\Z(p^\infty)$ a quotient of $A/T(A)$. 
Here $T(A)$ denotes the torsion group of $A$.
\end{enumerate}
\end{proposition}
For the definition of minimax groups, see section \ref{Characterization-abelian-groups-with-INI-family}.
If one applies part (2) of the previous proposition to the centre $Z$ of a finitely generated group $G$
one gets
\begin{corollary}
Let $G$ be a finitely generated group with centre $Z$.
Suppose there exists a prime $p$ 
such that $\Z(p^\infty)^2$ is a quotient of $Z$ and  $\Z(p^\infty)$ a quotient of $Z/T(Z)$. 
Then $G$ is an intrinsic condensation group.\qed
\end{corollary}

\begin{example}\label{powerabels}
For all $n\ge 3$, let  $A_n$ denote Abels' group.
Then the direct power $A_n^k$ is of intrinsic condensation for every $k > 1$. 
The group $A_n$, however, is not of intrinsic condensation 
(see Proposition \ref{CorCountQuot}). Thus the class of groups that are not of intrinsic condensation is not closed under direct products, not even under direct powers.
\end{example}

\subsection{Examples of extrinsic condensation groups}
\subsubsection{Largely related groups}\label{lrg}
Every infinitely presented metabelian groups is an extrinsic condensation group.
This claim will be established in Section \ref{s_bns}, 
using the notion of largely related groups, which we now introduce.

\begin{definition}\label{dlr}
We say that a finitely generated group $H$ is {\em largely related} if for every epimorphism $G\epi H$ of a finitely presentable group $G$ onto $H$, the kernel $N$ admits a non-abelian free quotient.
\end{definition}

Clearly, a largely related group cannot be finitely presentable; the converse is not true because the quotient of any finitely presentable group by an infinitely generated central subgroup is infinitely presented but not largely related.

\begin{proposition}\label{lr}
If a finitely generated group $H$ is largely related, then it is of extrinsic condensation and is not $(\text{FP}_2)$ over any nonzero commutative ring.
\end{proposition}
\begin{proof}
Let us first prove the second assertion. We have to prove that for every nonzero commutative ring $R$ and every epimorphism $G\epi H$ with $G$ finitely presentable and with kernel $N$, we have $N_\textnormal{ab}\otimes_\Z R\neq 0$. Indeed, since $N$ has a non-abelian free quotient, $N_\textnormal{ab}$ admits $\Z$ as a quotient and so $N_\textnormal{ab}\otimes_\Z R\neq 0$.

To check that $H$ is of extrinsic condensation, let $p:G\epi H$ be an epimorphism with $G$ finitely presentable and we have to show that the kernel $N$ of $p$ contains uncountably many subgroups that are normal in $G$. Since by assumption $N$ has a non-abelian free quotient, this follows from Corollary \ref{ontofree}.
\end{proof}

\subsubsection{Other examples of extrinsic condensation groups}
Thanks to Lemma \ref{indcon}(\ref{ind1}), infinitely independently presentable (INIP) groups are of extrinsic condensation. This fact is a first rich source of examples. 
Below we describe a further method of obtaining extrinsic condensation groups.  

We begin with an analogue of Proposition \ref{p_schur}.
\begin{proposition}\label{p_schure}
Let $G$ be a finitely generated group. 
Suppose that $H_2(G,\Z)$ is either not minimax 
or that it admits $\Z(p^\infty)^2$ as a quotient for some prime $p$. 
Then $G$ is of extrinsic condensation.
\end{proposition}
\begin{proof}
If $H_2(G,\Z)$ is not minimax the group $G$ is infinitely independently presented by Proposition \ref{p_schur} 
and so of extrinsic condensation by Lemma \ref{indcon}(1).

Suppose now that $H_2(G, \Z)$ maps onto $\Z(p^\infty)^2$.
Let $F$ be a finitely generated free group with an epimorphism $F\epi G$ with kernel $R$ and set $A=R/[R,F]$.
As in the proof of Proposition \ref{p_schur} 
one then sees that $A$ is isomorphic to a direct sum $H_2(G,\Z) \oplus B$ 
with $B$ a free abelian group of finite rank and so $\Z(p^\infty)^2$ is an image of $R/[R,F]$,
say $\rho \colon R/[R,F] \epi \Z(p^\infty)^2$. 
Suppose next that $R_1$ is a normal subgroup generated by a finite subset of $R$.
Then $\rho(R_1 \cdot [R,F]/[R,F])$ is a finite subgroup $K$ of $\Z(p^\infty)^2$ and so $\Z(p^\infty)^2/K$ is isomorphic to $\Z(p^\infty)^2$. 
The group $R/R_1$ maps therefore onto a group with $2^{\aleph_0}$ normal subgroups.
This proves that $G$ is an extrinsic condensation group.
\end{proof}

\begin{proposition}\label{gfpz}
Let $G$ be a finitely generated group and $Z \triangleleft G$ a central subgroup. 
\begin{enumerate}[(1)]
\item If $G$ is finitely presentable and $Z$ is minimax then $G/Z$ is not infinitely independently presentable. 
\item If $Z$ maps onto $\Z(p^\infty)^2$ for some prime $p$ then $G/Z$ is of extrinsic condensation.
\end{enumerate}
\end{proposition}
\begin{proof}
(1)  Let $F$ be a free group of finite rank, 
let $\pi \colon F \epi G$ be an epimorphism with kernel $R$
and $S \triangleleft F$ the preimage of $Z$ under $\pi$.
Assume $S$ is generated by an infinite family  $\{S_i \mid i \in I\}$ of normal subgroups.
Since $G$ is finitely presentable there exists a finite subset $J \subset I$ so that $R$ is contained in $S_J$.
If the family $\{S_i \mid i \in I\}$ were independent,
then $\{S_i \cdot R)/R \mid i \in I \smallsetminus J \}$ would be an infinite independent family of subgroups of the abelian group $Z$ and so $Z$ could not be minimax 
(by implication $(iv) \Rightarrow (i)$
of Lemma \ref{abq}). 

(2) The extension of groups $Z \triangleleft G \epi G/Z$ induces in homology an exact sequence
\[
H_2(G, \Z) \longrightarrow H_2(G/Z, \Z) \stackrel{\partial}{\longrightarrow} 
Z \stackrel{\res}{\longrightarrow} G_{\ab} \longrightarrow (G/Z)_{\ab} \to 0
\]
(see for instance \cite[Cor.~8.2]{HiSt97}).
It shows that $H_2(G/Z,\Z)$ maps onto the kernel $K$ of $\res \colon Z \to G_{\ab}$.
The centre $Z$ is therefore an extension of $K$ by a finitely generated group.
By hypothesis, there exists an epimorphism $\rho \colon Z \epi \Z(p^\infty)^2$.
Since $\Z(p^\infty)^2$ is a torsion group, the image of $\rho \circ \partial \colon K \to \Z(p^\infty)^2$ 
has finite index and so it is isomorphic to $\Z(p^\infty)^2$.
Therefore $H_2(G/Z,\Z)$ maps onto $\Z(p^\infty)^2$, 
and so $G/Z$ is of extrinsic condensation by Proposition \ref{p_schure}.
\end{proof}

\begin{example}\label{abelsnik}
If $n\ge 4$ and $k\ge 2$, and $A_n$ denotes Abels' group with centre $Z\simeq\Z[1/p]$, then the group $G=(A_n/Z)^k$ is of extrinsic condensation but not INIP by Proposition \ref{gfpz}.
\end{example}

\subsection{Counterexamples to different kinds of condensation}

For finitely presentable groups, condensation is equivalent to intrinsic condensation.
This fact yields many examples of non-condensation groups.
\begin{itemize}
\item Isolated groups. These include finite groups, finitely presentable simple groups, but also many groups having uncountably many normal subgroups. We refer the reader to \cite{CGP07} for a large family of examples.
\item Finitely presented groups satisfying max-n; 
these include, in particular, finitely presentable metabelian groups. 
\end{itemize}

It is more delicate to provide non-condensation infinitely presented groups. The first example is the following.

\begin{example}\label{abelsni}
If $n\ge 4$, the infinitely presented group $G=A_n/Z$ is not of condensation, 
a fact already observed in Corollary \ref{crl:Countable-CB-rank}
and in contrast with Example \ref{abelsnik}.
\end{example}

\begin{remark}\label{sapirem}
Example \ref{abelsni}, in conjunction with Corollary \ref{nominp}, 
provides examples of infinitely presented groups that are not INIP, and thus have no minimal presentation 

In the context of group varieties, Kleiman \cite{Kle81} constructed varieties of groups with no independent defining set of identities. It is however not clear whether the free groups (on $n\ge 2$ generators) in these varieties fail to admit a minimal presentation.
\end{remark}

\begin{example}\label{cnine}
There exist condensation groups that are
\begin{itemize}
\item \emph{extrinsic but not intrinsic}. The group $\Z\wr\Z$ is such an example. 
It is not of intrinsic condensation because it satisfies max-n 
and thus has countably many normal subgroups; 
on the other hand it is INIP by Example \ref{example:minimal} 
and therefore of extrinsic condensation by Lemma \ref{indcon}(\ref{ind1}).
\item \emph{intrinsic but not extrinsic}.
The free group $F_d$ is an example for $d\ge 2$. We already mentioned in \S\ref{sssfrg} that it is of intrinsic condensation. It is finitely presentable and thus is clearly not of extrinsic condensation. 
\item \emph{neither extrinsic nor intrinsic}. The simplest example we are aware of is much more elaborate 
and is based, once more, on Abels' group $A_4$ and its centre $Z\simeq\Z[1/p]$. 
We claim that $\Gamma=A_4\times (A_4/Z)$ is an example.

It is condensation because $\{0\}\times\Z[1/p]$ is a condensation point in $\mathcal{N}(\Z[1/p]^2)$ (this was shown in \cite{CGP10} but readily follows from the arguments in \cite{Boy56}).

It is not of extrinsic condensation, since $A_4\times A_4$ is finitely presentable 
and the kernel $\{1\} \times Z\simeq\Z[1/p]$ of the projection $A_4\times A_4 \epi \Gamma$ has only countably many subgroups.

Finally, $\Gamma$ is not of intrinsic condensation 
because it has only countably many normal subgroups, as a consequence of Lemma \ref{lemcos}. 
\end{itemize}
\end{example}

 
\section{Geometric invariant and metabelian condensation groups}\label{s_bns}

\subsection{Splittings}
\subsubsection{The splitting theorem}

Let $G$ be a group and $\pi \colon G \epi \Z$ an epimorphism.
We say that $\pi$ \emph{splits over a subgroup $U \subset G$} 
if there exists an element $t \in G$ and a subgroup $B \subseteq \ker \pi$ 
which contains both $U$ and $V = tUt^{-1}$, 
all in such a way that $G$ is canonically isomorphic to the HNN-extension $\HNN(B, \tau \colon U \iso V)$ 
with vertex group $B$, edge group $U$ and end point map $\tau$ given by conjugation by $t$.
Notice that $t$ plays the r\^ole of the stable letter. If $U\neq B\neq V$, we say that the splitting is {\em non-ascending}.

We are interested in the condition that $\pi$ splits over a \emph{finitely generated} subgroup.

An old result which is relevant in this context is Theorem A of \cite{BiSt78}; it asserts:
\emph{If $G$ is finitely presentable 
then every epimorphism $\pi \colon G \epi \Z$ splits over a finitely generated subgroup}.
We refine this old result to the following structure theorem:
\begin{theorem}
\label{thm:Structure-theorem}
Let $G$ be a finitely generated group and $\pi \colon G \epi \Z$ an epimorphism.
Then
\begin{enumerate}[(a)]
\item\label{st1} either $\pi$ splits over a finitely generated subgroup,
\item\label{st2} or $G$ is an inductive limit of a sequence of epimorphisms of finitely generated groups 
$G_n \epi G_{n+1}$ ($n\ge 0$) with the following features:
\begin{itemize}
\item the kernels of the limiting maps $\lambda_n \colon G_n \epi G$ are free of infinite rank, and
\item the epimorphisms $\pi \circ \lambda_n \colon G_n \epi \Z$ 
split over a compatible sequence of finitely generated subgroups, and have compatible stable letters. Moreover, the vertex groups are isomorphic to finitely generated subgroups of $\textnormal{Ker}(\pi)$.
\end{itemize}
\end{enumerate}
\end{theorem}

\begin{remark}
If a \emph{finitely generated} group $G$ splits over a finitely generated subgroup (either as an HNN-extension or as an amalgamated product) then the vertex groups are necessarily finitely generated, too.
\end{remark}

\subsubsection{Consequences}
\label{sssec:Corollaries-structure-theorem}

We first develop consequences and applications of Theorem \ref{thm:Structure-theorem}, which will be proved in \S\ref{ptst}.
\begin{corollary}
\label{crl:Corollary-for-ii}
Let $G$ be a finitely generated group 
with an epimorphism $\pi \colon G \epi \Z$ 
that does not split over a finitely generated subgroup.
Then $G$ is largely related (see Definition \ref{dlr} and \S\ref{lrg}). In particular,
\begin{enumerate}[(1)]
\item\label{ca} $G$ is not of type $(\text{FP}_2)$ over any non-zero commutative ring $K$; 
in particular, $G$ is not finitely presentable.
\item\label{cb} $G$ is an extrinsic condensation group.\qed
\end{enumerate}
\end{corollary}

To apply Corollary \ref{crl:Corollary-for-ii}, it is useful to clarify 
whether an epimorphism $G\epi\Z$ splits over a finitely generated subgroup. 
Such a splitting can be ascending or not, and the next two propositions deal with each of theses cases.

\begin{definition}
\label{definition:Contracting-automorphism} 
An automorphism $\alpha \colon N \longrightarrow N$  \emph{contracts into a subgroup} $B\subset N$ 
if $\alpha(B)\subset B$ and $N=\bigcup_{n\ge 0}\alpha^{-n}(B)$.
\end{definition}

\begin{proposition}
Consider an epimorphism $\pi:G\epi\Z$ with kernel $N$. If $t\in G$, let $\alpha_t$ be the automorphism of $N$ defined by $\alpha_t(g)=tgt^{-1}$. The following statements are equivalent:
\begin{enumerate}[(i)]
\item\label{as_1} $\pi$ has an ascending splitting over a finitely generated subgroup;
\item\label{as_2} for some $t\in\pi^{-1}(\{1\})$, either $\alpha_t$ or $\alpha_t^{-1}$ contracts into a finitely generated subgroup of $N$;
\item\label{as_3} for every $t\in\pi^{-1}(\{1\})$, either $\alpha_t$ or $\alpha_t^{-1}$ contracts into a finitely generated subgroup of $N$.
\end{enumerate}
\end{proposition}
\begin{proof}
(\ref{as_1})$\Leftrightarrow$(\ref{as_2})$\Leftarrow$(\ref{as_3}) is obvious. Suppose (\ref{as_2}) and let us prove (\ref{as_3}). 
So there exists an element $t_0 \in G$ with $\pi(t_0) = 1$ and a sign $\varepsilon\in\{\pm 1\}$ such that $\alpha_{t_0}^\varepsilon$ contracts into a finitely generated subgroup of $N$. Let us check that for every $t\in G$ with $\pi(t)=1$, $\alpha_{t}^\varepsilon$ contracts into a finitely generated subgroup of $N$.

By assumption, there exists a finitely generated subgroup $M\subset N$ such that 
 \[
 M \subseteq \alpha^{-\varepsilon}_{t_0} (M) \text{ and } \bigcup\nolimits_{j\ge 0} \alpha^{-\varepsilon j} _{t_0} (M)= N.
 \]
 There exists an element  $w \in N$ with $t = t_0 \cdot w$,
 and hence there exists a non-negative integer $k$ with $w \in M' = \alpha^{-\varepsilon k}_{t_0}(M)$.
Then $M'$ is a finitely generated subgroup of $N$ and $\alpha^{\varepsilon}_{t}$ contracts into it.
\end{proof}

For the case of a non-ascending splitting, there is no characterization other than obvious paraphrases. The following proposition provides a (non-comprehensive) list of necessary conditions, which by contraposition provide obstructions to the existence of a non-ascending splitting.

\begin{proposition}\label{cnas}
Consider an epimorphism $\pi:G\epi\Z$ with kernel $N$, with a non-ascending splitting over a finitely generated subgroup. Then each of the following properties holds.  
\begin{enumerate}[(1)]
\item\label{nas12} $N$ is an iterated amalgam $$\cdots A_{-2}\ast_{B_{-2}}A_{-1}\ast_{B_{-1}}A_{0}\ast_{B_{0}}A_{1}\ast_{B_{1}}
 A_2\cdots$$
with all $B_i$ finitely generated and all embeddings proper (if $G$ is finitely generated then all $A_i$ can be chosen finitely generated);
\item\label{nas2} $N$ is an amalgam of two infinitely generated groups over a finitely generated subgroup;
\item\label{nas3} (assuming that $G$ is finitely generated) there exists $g\in N$ whose centralizer in $N$ is contained in a finitely generated subgroup of $N$;
\item\label{nas5} there exists an isometric action of $N$ on a tree with two elements acting as hyperbolic isometries with no common end, with finitely generated edge stabilizers;
\item\label{nas4} $G$ has a non-abelian free subgroup;
\item\label{nasp} $N$ is not the direct product of two infinitely generated subgroups.
\end{enumerate}
\end{proposition}
\begin{proof}
Denote by (*) the property that $\pi$ has a non-ascending splitting over a finitely generated subgroup. We are going to check (*)$\Rightarrow$(\ref{nas12})$\Rightarrow$(\ref{nas2})$\Rightarrow$(\ref{nas3}), (\ref{nas2})$\Rightarrow$(\ref{nas5})$\Rightarrow$(\ref{nas4}), and (\ref{nas5})$\Rightarrow$(\ref{nasp}).

\noindent (*)$\Rightarrow$(\ref{nas12}) If $\pi$ has a non-ascending splitting $\HNN(B, \tau \colon U \iso V)$, then $N$ can be described \cite[\S 1.4]{Ser77b} as the amalgamation 
$$ \cdots B\ast_U B \ast_U B \ast_U B\cdots$$
where all left embeddings $U\mono B$ are the inclusion, and all right embeddings are given by $\tau$. So (\ref{nas12}) holds. If $G$ is finitely generated, then $B$ is finitely generated as well.

\noindent (\ref{nas12})$\Rightarrow$(\ref{nas2})
If $N$ has the form (\ref{nas12}), then in particular $N$ is the amalgam
$$ (\cdots A_{-2}\ast_{B_{-2}}A_{-1}\ast_{B_{-1}} A_0) \ast_{B_0} (A_1 \ast_{B_1}A_2\ast_{B_2} A_3\cdots),$$
where both vertex groups are infinitely generated and the edge group $B_0$ is finitely generated.

\noindent (\ref{nas2})$\Rightarrow$(\ref{nas3}) Suppose that $N=A*_BC$ with $B$ finitely generated and $A\neq B\neq C$. Let $N$ act on its Bass-Serre tree, and let $w$ be some hyperbolic element (e.g.\ $ac$ for $a\in A\smallsetminus B$ and $c\in C\smallsetminus B$), with axis $D$. Let $H$ be the centralizer of $w$. Observe that $H$ preserves the axis $D$ of $w$. Let $H_1$ be the kernel of the action of $H$ on $D$. So $H_1$ is contained in a conjugate of $C$, say the finitely generated subgroup $gCg^{-1}$. Since $H/H_1$ is a group of isometries of the linear tree, it has at most two generators, so $H$ is also contained in a finitely generated subgroup of $N$.

\noindent (\ref{nas2})$\Rightarrow$(\ref{nas5})$\Rightarrow$(\ref{nas4}). 
The implication (\ref{nas2})$\Rightarrow$(\ref{nas4}) is explicitly contained in \cite[Thm.~6.1]{Bas76}, 
but both implications are folklore and have, at least in substance been proved many times or in broader contexts. 
A direct proof of (\ref{nas5})$\Rightarrow$(\ref{nas4}) can be found in \cite[Lem.~2.6]{CuMo87}. 
Let us justify (\ref{nas2})$\Rightarrow$(\ref{nas5}). 
For an amalgam, an easy argument (see for instance \cite[Lem.~12 and 14]{dCor09}) shows that the action of an amalgam of two groups over subgroups of index at least 2 and 3 respectively on its Bass-Serre tree does not fix any point in the 1-skeleton, nor any end or pair of ends. 
By \cite[Thm.~2.7]{CuMo87}, there are two elements acting as hyperbolic isometries with no common end.

\noindent (\ref{nas5})$\Rightarrow$(\ref{nasp}) Suppose that $N=N_1\times N_2$ with $N_1$, $N_2$ both infinitely generated. Supposing (\ref{nas5}), we can suppose that the action is minimal (otherwise restrict to the convex hull of all axes of all hyperbolic elements). Since the product of two commuting elliptic isometries is elliptic, either $N_1$ or $N_2$, say $N_1$, contains an element $w$ acting as a hyperbolic isometry. Since $w$ centralizes $N_2$, the axis $D$ of $w$ is $N_2$-invariant. If $N_2$ contains a hyperbolic element, it follows that this axis is $D$, so by repeating the argument, $N_1$ in turn stabilizes $D$, so $N$ stabilizes $D$, a contradiction. Therefore $N_2$ acts by elliptic isometries. So it either fixes a point in the 1-skeleton, or fixes a unique point at infinity. In the latter case, this point is stabilized by all of $N$, a contradiction; in the former case, the set of $N_2$-fixed points is stable and therefore by minimality is all of the tree, i.e., $N_2$ acts trivially. So the edge stabilizers contain $N_2$ and therefore are infinitely generated, a contradiction.
\end{proof}

\begin{example}\label{sla}
We provide here a nontrivial example where Corollary \ref{crl:Corollary-for-ii} applies although $N$ contains non-abelian free subgroups. Fix $d\ge 3$. Consider the ring $A=\Z[x_n:n\in\Z]$ of polynomials in infinitely many variables. Let the group $\Z$ act on $A$ by ring automorphisms so that the generator $1\in\Z$ sends $x_n$ to $x_{n+1}$. This ring automorphism induces an automorphism $\phi$ of the subgroup $\textnormal{EL}_d(A)\subset\SL_d(A)$ generated by elementary matrices, i.e.\ those matrices $e_{ij}(a)$ with $a\in A$ and $1\le i\neq j\le d$, whose entries differ from the identity matrix only at the $(i,j)$ entry, which is $a$. (Actually $\textnormal{EL}_d(A)=\SL_d(A)$ by a nontrivial result of Suslin \cite{Sus77}, but we do not make use of this fact, except in the statement of Corollary \ref{isla}.)

We are going to check that $\EL_d(A)\rtimes_\phi\Z$ is finitely generated, and that $\pi:\EL_d(A)\rtimes\Z\epi\Z$ does not split over a finitely generated subgroup. So Corollary \ref{crl:Corollary-for-ii} applies; in particular, $\EL_d(A)\rtimes_\phi\Z$ is an extrinsic condensation group.

Let us first check that $\EL_d(A)\rtimes_\phi\Z$ is finitely generated, or equivalently that $\EL_d(A)$ is finitely generated as a $\langle\phi\rangle$-group. By definition, $\EL_d(A)$ is generated by elementary matrices; since $e_{ij}(a)e_{ij}(b)=e_{ij}(a+b)$, those elementary matrices for which $a$ is a monomial are enough; since $e_{ij}(ab)=[e_{ik}(a),e_{kj}(b)]$ (and $d\ge 3$), those for which $a=x_n$ for some $n$ are enough. This shows that $\EL_d(A)$ is generated, as $\langle\phi\rangle$-group, by $\{e_{ij}(x_0):1\le i\neq j\le d\}$. So $\EL_d(A)\rtimes_\phi\Z$ is finitely generated.

Clearly $\phi$ does not contract into a finitely generated subgroup, as it satisfies the following totally opposite property: for every finitely generated subgroup $M\subset\EL_d(A)$ and every $x\in\EL_d(A)\smallsetminus \{1\}$, the set $\{n\in\Z:\phi^n(x)\in M\}$ is finite.

To check that $\pi$ does not admit any non-ascending splitting, let us check that (\ref{nas5}) of Proposition \ref{cnas} fails. Indeed, we have $N=\bigcup_{n\ge 0}\EL_d(A_n)$ with $A_n=\Z[x_i:-n\le i\le n]$ and it is known \cite[p.~681]{CuVo96} that $\EL_d(A_n)$ has Property (FA) (every action on a tree has a fixed point) for all $d\ge 3$. In particular, every isometric action of $\EL_d(A)$ on a tree is by elliptic isometries, contradicting (\ref{nas5}) of Proposition \ref{cnas}.

A variant of the result is obtained by considering $\GL_\Z(R)$, the group of infinite matrices with entries in a fixed finitely generated ring $R$ (with unit $1\neq 0$, possibly non-commutative), indexed by $\Z\times\Z$ and that differ from identity matrix entries for finitely many indices only, and its subgroup $\EL_\Z(R)$ generated by elementary matrices. Let $\psi$ be the automorphism of $\EL_\Z(R)$ shifting indices (so that $\psi(e_{ij}(a))=e_{i+1,j+1}(a)$). Using arguments similar to those of the previous construction, $\EL_\Z(R)\rtimes_\psi\Z$ is finitely generated, $\psi^{\pm 1}$ does not contract into a finitely generated subgroup. Besides, $\EL_\Z(R)$ fails to satisfy (\ref{nas3}) of Proposition \ref{cnas}. So $\EL_\Z(R)\rtimes_\psi\Z$ is an extrinsic condensation finitely generated group. If the K-theory abelian group $K_1(R)\simeq\GL_\Z(R)/\EL_Z(R)$ is finitely generated, then the argument also works for $\GL_\Z(R)\rtimes_\psi\Z$.
\end{example}

\subsubsection{Proof of Theorem \ref{thm:Structure-theorem}}\label{ptst}
Choose an element $t \in G$  with $\pi(t) = 1$.
Since $G$ is finitely generated, there exists a finite subset $S$ in  $N = \ker \pi$
such that $\bigcup_{n \in \Z} t^nSt^{-n}$ generates $N$. 

For $n\ge 0$, 
let $B_n$ be the subgroup generated by the finite set $\bigcup_{0 \leq j \leq n} t^jSt^{-j}$; for $n\ge 1$, set $U_n = B_{n-1}$ and $V_n = t U_n t^{-1}$; note that both $U_n$ and $V_n$ are contained in $B_n$. We also set $U_0=V_0=\{1\}$.
For $n\ge 0$, define $G_n$ to be the HNN-extension
\[
\HNN(B_n, t_n \mid \tau_n \colon U_n \iso V_n),
\]
the isomorphism $\tau_n$ being given by conjugation by $t$ in $G$. 
Then there exist unique epimorphisms $G_n \epi G_{n+1} \epi G$ 
that extend the inclusions $B_n \incl B_{n+1} \incl N$ and map $t_n$ to $t_{n+1}$ and $t_{n+1}$ to $t$.
These epimorphisms are clearly compatible with the HNN-structure of the groups involved; in particular, there is a natural surjective homomorphism $\phi:\underrightarrow{\lim}\,G_n\epi G$.

Let $W_n$ denote the kernel of the limiting map $\lambda_n \colon G_n \epi G$ 
and consider the action of $G_n$ on its Bass-Serre tree $T$. 
Since $\lambda_n$ is injective on the vertex group $B_n$,
the action of $W_n$ on $T$ is free and so $W_n$ is a free group.

We claim that $\phi$ is an isomorphism. This amounts to showing that every element $w$ in the kernel of $G_0\epi G$ lies in the kernel of $G_0\epi G_n$ for $n$ large enough. Write $w=t^{n_1}s_1\dots t^{n_k}s_k$ with $n_i\in\Z$ and $s_i\in S^{\pm 1}$. Clearly $\sum n_i=0$, so $w$ can be rewritten as $\prod_{i=1}^kt^{m_i}s_it^{-m_i}$ for integers $m_i=n_1+\dots + n_i$. After conjugation by some power of $t$, we can suppose that all $m_i$ are positive; let $m$ be the maximum of the $m_i$. Then, in $G_m$, $w$ belongs to $V_m$, which is sent one-to-one to $G$. Therefore $w=1$ in $G_m$.

Suppose now that for some $n$ the kernel $W_{n}$ is generated, as a normal subgroup, by a finite subset $\FF$. 
Since $G$ is the inductive limit of the $G_k$ there exists then an index $m \geq n$ 
such that the canonical image of $\FF$ under the canonical epimorphism $G_{n} \epi G_m$ 
lies in the trivial subgroup $\{1 \}$ of $G_m$. 
The canonical map $G_{n} \epi G_m$ induces therefore an epimorphism $G_{n}/W_{n} \epi G_m$.
It follows that the composition $G_n/W_n \epi G_m \epi G$ is an isomorphism, 
whence $\lambda_m \colon G_m \epi G$  is bijective and so statement (\ref{st1}) holds.
If, on the other hand, each kernel $W_n$ is infinitely generated as a normal subgroup, 
assertion (\ref{st2}) is satisfied.

\subsection{The geometric invariant $\Sigma(G)$}\label{subns}

The question whether and how an epimorphism  $\pi \colon G\epi \Z$ splits over a finitely generated subgroup 
can profitably be investigated with the help of the invariant $\Sigma(G)$ introduced in \cite{BNS87}.

\begin{definition}
\label{definition:Sigma(G)}
 $\Sigma(G)$ is the conical subset of $\Hom(G, \R)$ 
 consisting of all \emph{non-zero} homomorphisms $\chi \colon G \to \R$ with the  property 
 that the commutator subgroup $G'$ is finitely generated as $P$-group
 for some finitely generated submonoid $P$ of $G_\chi=\{g \in G \mid \chi(g)\geq 0\}$.
 
 The complement of $\Sigma(G)$ in $\Hom(G, \R)$ 
 will be denoted by $\Sigma^c(G)$.
 \end{definition}
In \cite{BNS87}, the invariant is denoted by $\Sigma_{G'}(G)$ 
and is defined as the projection of the above to the unit sphere, or equivalently the set of rays of the vector space $\Hom(G,\R)$. Since $\Sigma(G)$ is stable by multiplication by positive real numbers, these are obviously equivalent data.

 It is an important non-trivial fact that $\Sigma(G)$ is an open subset of $\Hom(G, \R)$.
 
 There are various equivalent definitions of $\Sigma(G)$ 
and there exists a restatement of the condition $\chi \in \Sigma(G)$  in the case  where $\chi(G) = \Z$
that is particularly concise. 
Indeed, the equivalence of (i) and (iii) in \cite[Prop.~4.3]{BNS87} yields
\begin{proposition}
\label{prp:Link-with-BS78}
Let $G$ be a finitely generated group, $\chi \colon G \epi \Z$ an epimorphism, and $t\in G$ with $\chi(t)=1$. 
Then $\chi \in \Sigma(G)$ if and only if  
the action of $t^{-1}$ on $N=\ker \chi$ contracts into a finitely generated subgroup.
\end{proposition}

\subsubsection{Application to extrinsic condensation groups}
\label{sssec:Application-extrinsic-condensation}
Definition \ref{definition:Sigma(G)} implies 
that $\Sigma(G)$ is invariant under multiplication by positive real numbers.
Proposition \ref{prp:Link-with-BS78} allows us therefore to deduce from Corollary \ref{crl:Corollary-for-ii}
the following result:
\begin{corollary}
\label{crl:Criterion-for-extrinsic-condensation}
Let $G$ be a finitely generated group.
If $\Sigma^c(G)$ contains a rational line $\R\chi$ with $\chi=G\epi\Z$ such that $\chi$ does not have a non-ascending splitting, then $G$ is an extrinsic condensation group. 
\end{corollary}

\begin{example}
Let $G_1,G_2$ be finitely generated groups, and let $\chi_i:G_i\epi\Z$, $i=1,2$ be epimorphisms. Assume that $\chi_i\in\Sigma^c(G_i)$ for $i=1,2$. Given two positive integers $m_1,m_2$, define $G$ as the fibre product
$$G=\{(g_1,g_2)\in G_1\times G_2)\mid m_1\chi_1(g_1)+m_2\chi_2(g_2)=0\}.$$
Then $G$ is finitely generated and is an extrinsic condensation group, and in particular is infinitely presented. Indeed, if on $G$ we define $\chi(g_1,g_2)=\chi_1(g_1)$, then it follows from the definition that the rational line $\R\chi$ is contained in $\Sigma^c(G)$. On the other hand, we have $\textnormal{Ker}(\chi)=\textnormal{Ker}(\chi_1)\times\textnormal{Ker}(\chi_2)$ and both factors are infinitely generated because $\chi_i\in\Sigma^c(G_i)$, so by Proposition \ref{cnas}(\ref{nasp}), we deduce that $G$ does not have a non-ascending splitting over $\chi$. Thus all assumptions of Corollary \ref{crl:Criterion-for-extrinsic-condensation} are satisfied. 
This example generalizes \cite[Theorem 6.1]{BrMi09}.
\end{example}

\begin{remark}
\label{remark:Polyhedrality}
In all cases where $\Sigma^c(G)$ has been determined explicitly
the invariant has turned out to be a \emph{polyhedral} cone.
And with the exception of certain groups of PL-homeomorphisms (see \cite[\S 8]{BNS87})
all known subsets are, in fact, \emph{rational-poly\-hedral}.
In these cases, the condition ``$\Sigma^c(G)$ \emph{contains a rational line}'' is, of course,
equivalent to ``$\Sigma^c(G)$ \emph{contains a line}''.
\end{remark}

\subsection{Metabelian groups and beyond}
\label{ssec:Metabelian-groups-and-beyond}
We need the following two major results about the $\Sigma$-invariant of the pre-\cite{BNS87}-era:
Let $G$ be a finitely generated metabelian group.
Then
\begin{enumerate}[(1)]
\item\label{imp1} $G$ is infinitely presented if, and only if, $\Sigma^c(G)$ contains a line (Theorem A(ii) in \cite{BiSt80}).
\item\label{imp2} $\Sigma^c(G)$ is rational polyhedral (Theorem E in \cite{BiGr84}).
\end{enumerate}
In view of Corollary \ref{crl:Criterion-for-extrinsic-condensation} and the remark following it
these two results imply
that \emph{a finitely generated metabelian group is of extrinsic condensation 
if, and only, if it is not finitely presentable}.

Parts of this characterization of metabelian, extrinsic condensation groups 
are also valid for larger classes of solvable groups. To gain the proper perspective,
we begin with a general result.
If $G$ is a group, denote by $G'$ its derived group and $\gamma_3(G') =[[G',G'],G']$ the third term of the lower central series of $G'$, so that $G/\gamma_3(G')$ is the largest (2-nilpotent)-by-abelian quotient of $G$.

\begin{theorem}\label{prp:Results-for-solvable-groups}
Let $G$ be a finitely generated solvable group.
Then the following statements hold:
\begin{enumerate}[(1)]
\item\label{fsa} if $G$ is finitely presentable, so is its metabelianization $G/G''$;
\item\label{fsb} if $G/G''$ is finitely presentable, the nilpotent-by-abelian quotient $G/\gamma_3(G')$ 
satisfies max-n and in particular is not of intrinsic condensation;
\item\label{fsc} if $G/G''$ is not finitely presentable, then $G$ is largely related and in particular is of extrinsic condensation.
\end{enumerate}
\end{theorem}
\begin{proof}
Claim (\ref{fsa}) is a special case of \cite[Thm.~B]{BiSt80}. 
If $G/G''$ is finitely presentable, then $G/\gamma_3(G')$ satisfies max-n by \cite[Thm.~5.7]{BiSt80},
so it has only countably many normal subgroups and thus (\ref{fsb}) holds.
Assume now that  $G/G''$ does not admit a finite presentation.
By properties  (1) and (2) stated in the above,
the invariant $\Sigma^c(G/G'')$ contains then a rational line. 
An easy consequence of definition \ref{definition:Sigma(G)} is that $\Sigma^c(G/G'')\subset\Sigma^c(G)$, so $\Sigma^c(G)$ contains a rational line as well.
Claim (\ref{fsc}) now follows from Corollary \ref{crl:Criterion-for-extrinsic-condensation}.
\end{proof}

\begin{remark}
\label{remark:Sigma-extrinsic-condensation}
Let $G$ be a finitely generated (2-nilpotent)-by-abelian group (i.e.\ with $\gamma_3(G') = \{1\}$).
If the metabelianization $G/G''$ is finitely presentable, then $G$ is not of intrinsic condensation by Claim (\ref{fsb})
in the preceding proposition.
Whether or not it is of (extrinsic) condensation depends, however, on the particular group.
To see this consider Abels' group $A_4$ studied in \S\ref{ssec:Abels-groups}. Then $B=A_4/\ZZ(A_4)$ is (2-nilpotent)-by-abelian, and so is $B\times B$, and their metabelianizations are finitely presentable; however, by Examples \ref{abelsnik} and \ref{abelsni}, $B\times B$ is of extrinsic condensation but $B$ is not (and is not of condensation either).
\end{remark}

\begin{corollary}
\label{crl:Fp-centre-by-metabelian-groups}
For every finitely generated centre-by-metabelian group $G$
the following statements are equivalent:
\begin{enumerate}[(i)]
\item\label{cm1} $G$ is not finitely presentable;
\item\label{cm2} $G/G''$ is not finitely presentable;
\item\label{cm3} $\Sigma^c(G)$ contains a rational line;
\item\label{cm4} $G$ is largely related;
\item\label{cm5} $G$ is of extrinsic condensation;
\item\label{cm6} $G$ is of condensation.
\end{enumerate}
\end{corollary}

\begin{proof}
If $G/G''$ is finitely presentable the central subgroup $G''$ is finitely generated 
by (\ref{fsb}) of Theorem \ref{prp:Results-for-solvable-groups}
and so $G$ itself is finitely presentable; implication (\ref{cm1}) $\Rightarrow $(\ref{cm2}) therefore holds.
Implication (\ref{cm2}) $\Rightarrow $(\ref{cm3}) is a consequence of properties (\ref{imp1}) and (\ref{imp2}) 
stated at the beginning of this \S\ref{ssec:Metabelian-groups-and-beyond}
and implication (\ref{cm3}) $\Rightarrow $(\ref{cm4}) a consequence of Corollary \ref{crl:Criterion-for-extrinsic-condensation}.
Implication (\ref{cm4}) $\Rightarrow $(\ref{cm5}) and (\ref{cm5}) $\Rightarrow $(\ref{cm6}) are true for arbitrary $G$, see Proposition \ref{lr} and Lemma \ref{condei}(\ref{ec}).
For (\ref{cm6}) $\Rightarrow $(\ref{cm1}), by contraposition if $G$ is finitely presentable, we have by Lemma \ref{condei}(\ref{ic}) to check that $G$ is not of intrinsic condensation, and this holds by Theorem \ref{prp:Results-for-solvable-groups}(\ref{fsb}).
\end{proof}

\subsection{Thompson's group $F$}\label{parat}

Recall that a group $G$ is called {\em finitely discriminable} if $\{1\}$ is isolated in $\mathcal{N}(G)$, or equivalently if there exists a finite family $N_1,\dots,N_k$ of nontrivial normal subgroups of $G$ such that every nontrivial normal subgroup of $G$ contains one of the $N_i$. An elementary observation in \cite{CGP07} is that a finitely generated group $G$ is isolated in $\mathcal{G}_m$ (for any marking) if and only if it is both finitely presentable and finitely discriminable.

We now turn to the proof of Corollary \ref{crl:Normal-subgroups-in-F}. In this paragraph, $F$ denotes Thompson's group of the interval as defined in the introduction. This is a finitely presentable group whose derived subgroup $[F,F]$ is an infinite simple group (see \cite{CFP96} for basic properties of $F$).

\begin{lemma}\label{ntpfd}
Let $N$ be any normal subgroup of $F$. Then $N$ is finitely discriminable.
\end{lemma}
\begin{proof}
We repeatedly use the fact that every nontrivial normal subgroup of $F$ contains the simple group $[F,F]$ (see \cite{CFP96}). This implies in particular that the centralizer of $[F,F]$ is trivial.

The result of the lemma is trivial if $N=1$. Otherwise, $N$ contains $[F,F]$ (see \cite{CFP96}). If $M$ is a nontrivial normal subgroup of $N$, since $[F,F]$ has trivial centralizer, it follows that $M\cap [F,F]$ is nontrivial and hence, by simplicity of $[F,F]$, $M$ contains $[F,F]$. This shows that $N$ is finitely discriminable.
\end{proof}

Let $N$ be a normal subgroup of Thompson's group $F$ such that $F/N$ is infinite cyclic. If in Corollary \ref{crl:Normal-subgroups-in-F}, we have $r=0$, this means that $N$ is equal of $\textnormal{Ker}(\chi_0)$ or $\textnormal{Ker}(\chi_1)$. Otherwise there exists two coprime nonzero integers $p,q$ such that $N$ is equal to the group $$N_{p,q}=\{f\in F: p\chi_0(f)=q\chi_1(f)\}=\textnormal{Ker}(p\chi_0-q\chi_1).$$

\begin{lemma}\label{nfg}
The groups $\textnormal{Ker}(\chi_0)$ and $\textnormal{Ker}(\chi_1)$ are infinitely generated, while for any nonzero coprime integers $p,q$, the group $N_{p,q}$ is finitely generated.
\end{lemma}
\begin{proof}
Observe that $\textnormal{Ker}(\chi_0)$ is the increasing union of subgroups $\bigcup\{g\in F:g_{|[0,1/n]}=\textnormal{Id}\}$ and $\textnormal{Ker}(\chi_1)$ can be written similarly; thus both groups are infinitely generated.

Now let $p,q$ be nonzero coprime integers. Since $N=N_{p,q}$ acts 2-homogeneously (i.e., transitively on unordered pairs) on $X=]0,1[\cap\Z[1/2]$ (because its subgroup $[F,F]$ itself acts 2-homogeneously), the stabilizer $N_{1/2}$ of $1/2$ is maximal in $N$ and therefore it is enough to check that $N_{1/2}$ is finitely generated.

Let $i_0,i_1$ be the following endomorphisms of $F$: if $f\in F$, $i_0(f)$ 
(resp. $i_1(f)$) acts on $[0,1/2]$ (resp.\ $[1/2,1]$) as $f$ acts on $[0,1]$ 
(i.e.\ $i_0(f)(t)=f(2t)/2$ for $0\le t\le 1/2$
 and $i_1(f)(t)=1/2+f(2t-1)/2$ for $1/2\le t\le 1$); while $i_0(f)$ is the identity on $[1/2,1]$ and $i_1(f)$ is the identity on $[0,1/2]$.
 
Define $F_k=\{g\in F:\chi_0(f)\in k\Z\}$ and similarly $F^k=\{g\in F:\chi_1(f)\in k\Z\}$; these are finite index subgroups of $F$ and are thus finitely generated. Fix $\sigma\in F$ with slope 2 at 0 and 1. Define, for $f\in F_q$, $j_0(f)=i_0(f)i_1(\sigma)^{p\chi_0(f)/q}$. By construction, $j_0(f)\in N_{1/2}$. Similarly, for $f\in F^p$, define $j_1(f)=i_0(\sigma^{q\chi_1(f)/p})i_1(f)$; then $j_1(f)\in N_{1/2}$.

We claim that $N_{1/2}$ is generated by $j_0(F_q)\cup j_1(F^p)$. Indeed, if $g\in N_{1/2}$, then $\chi_0(g)\in q\Z$, so by composition by an element in $j_0(F_q)$ we obtain an element which is the identity on $[0,1/2]$. In turn, by composition by an element of $j_1(F^p)$, we obtain an element which is the identity on $[1/2,1]$ and is a certain power $i_0(\sigma)^k$ on $[0,1/2]$; but since we obtained an element of $N$, necessarily $k=0$. This proves the claim and thus $N_{1/2}$, and hence $N$, are finitely generated.
\end{proof}

\begin{remark}\label{rcp}
There is a more conceptual proof of Lemma \ref{nfg}, based on the geometric invariant. It uses \cite[Thm.~B1]{BNS87}, which says that if $N$ is a normal subgroup of a finitely generated group $G$, then $N$ is finitely generated if and only if $\Sigma^c(G)\cap\Hom(G/N,\R)=\{0\}$, where $\Hom(G/N,\R)$ denotes the set of homomorphisms $G\to\R$ vanishing on $N$.
 
A simple verification \cite[Thm.~8.1]{BNS87} shows that $\Sigma^c(F)$ consists of the two half-lines generated by $\chi_0$ and $\chi_1$. This yields the statement of Lemma \ref{nfg}.
\end{remark}

\begin{lemma}
If $pq>0$ then $N_{p,q}$ is isomorphic to an ascending HNN-extension of $F$ and is an isolated (hence finitely presentable) group.
\end{lemma}
\begin{proof}
Set $N=N_{p,q}$ and let $t\in N$ generate $N/[F,F]$, so that $\chi_0(t)$ is positive (hence $\chi_1(t)>0$ as well). Then there exists a dyadic segment $I$, contained in $]0,1[$, such that $t(I)\subset I$ and $\bigcup_{n\ge 0}t^{-n}(I)=]0,1[$ (if $\alpha>0$ is a small enough dyadic number, then $[\alpha,1-\alpha]$ is such a segment). If $F(I)$ is the Thompson group in the interval $I$, then $tF(I)t^{-1}=F(t(I))\subset F(I)$ and $\bigcup_{n\ge 0}t^{-n}F(I)t^n=[F,F]$, and thus $N$ is the ascending HNN-extension with stable letter $t$ and vertex group $F(I)$. In particular, $N$ is finitely presentable. Since by Lemma \ref{ntpfd} it is finitely discriminable, it follows that $N$ is isolated.
\end{proof}

\begin{lemma}\label{bnsth}
If $pq<0$ then $\Sigma^c(N_{p,q})=\Hom(N_{p,q},\R)=\R$.
\end{lemma}
\begin{proof}
Set $N=N_{p,q}$. Let us check that $\chi=(\chi_0)_{|N}$ belongs to $\Sigma^c(N)$. 
We have to check that for any finite subset $S$ of $F_\chi$ generating a submonoid $M_S$, $[F,F]$ is not finitely generated as an $M_S$-group. There exists $\varepsilon>0$ such that for every $s\in S$, the element $s$ is linear in the interval $[0,\varepsilon]$ with slope at least one. It follows that if $H_n$ is the set of elements of $[F,F]$ that are equal to 1 in the interval $[0,2^{-n}\varepsilon]$, then $H_n$ is an $M_S$-subgroup. Since $[F,F]$ is the increasing union $\bigcup H_n$, it is therefore not finitely generated as an $M_S$-subgroup. 

The same argument shows that $(\chi_1)_{|N}$ belongs to $\Sigma^c(N)$. Since $pq<0$, $\Hom(N,\R)$ is the union of the half-lines generated by $(\chi_0)_{|N}$ and $(\chi_1)_{|N}$, and we are done.
\end{proof}

\begin{remark}
We do not use this fact, but the reader can check that if $pq>0$ and $N=N_{p,q}$ then $\Sigma^c(N)$ is the half-line generated by $(\chi_0)_{|N}$ (which also contains $(\chi_1)_{|N}$).
\end{remark} 

\begin{corollary}
If $pq<0$ then $N_{p,q}$ is an extrinsic condensation group and in particular is infinitely presented.
\end{corollary}
\begin{proof}
By Lemma \ref{nfg}, $N_{p,q}$ is finitely generated. So Lemma \ref{bnsth} together with Corollary \ref{crl:Criterion-for-extrinsic-condensation} imply that $N_{p,q}$ is of extrinsic condensation (and thus infinitely presented).
\end{proof}

\bibliographystyle{amsalpha}%
\bibliography{BCGSbi}

\end{document}